\documentclass{amsart}
\usepackage{gimac}
\usepackage{gicomments}
\usepackage{customs_Yue}
\RequirePackage{datetime,currfile}

\usepackage{pdfsync,hyperref}

\newcommand{\gobble}[1]{}

\newcommand{\blue}[1]{{#1}}  

\newcommand{\inv}{^{-1}}

\newcommand{\eps}{\epsilon}

\newcommand{\D}{\partial}
\newcommand{\Dt}{\partial_t}
\newcommand{\Dx}{\partial_x}

\newcommand{\calS}{{\mathcal S}}

\newcommand{\eep}{\eee^\eps}
\newcommand{\qqp}{\qqq^\eps}

\newcommand{\hide}[1]{}  

\newcommand{\hsstar}{h_{\star\star}}
\newcommand{\Tsstar}{T_{\star\star}}

\newcommand{\Clim}{C_{\rm lim}}

\begin{document}

\title[Blow-up for a dispersionless shallow water system]
{Well-posedness and derivative blow-up for a dispersionless regularized shallow water system}

\author{Jian-Guo Liu}
\address{
Department of Physics and Department of Mathematics\\
Duke University, Durham, NC 27708, USA}
\email{jliu@phy.duke.edu}
\author[R. L. Pego]{Robert L. Pego} 
\address{Department of Mathematical Sciences and Center for Nonlinear Analysis\\
Carnegie Mellon University, Pittsburgh, Pennsylvania, PA 12513, USA}
\email{rpego@cmu.edu}
\author[Y. Pu]{Yue Pu}
\address{Department of Mathematical Sciences and Center for Nonlinear Analysis\\
Carnegie Mellon University, Pittsburgh, Pennsylvania, PA 12513, USA}
\email{flamesofmoon@gmail.com}

\date{\today} 




\begin{abstract} 
We study local-time well-posedness and breakdown for solutions
of regularized Saint-Venant equations (regularized classical shallow water equations)
recently introduced by Clamond and Dutykh.  
The system is linearly non-dispersive, and
smooth solutions conserve an $H^1$-equivalent energy.
No shock discontinuities can occur, 
but the system is known to admit weakly singular shock-profile solutions
that dissipate energy.  We identify a class of small-energy
smooth solutions that develop singularities in the first derivatives in finite time. 
\end{abstract}

\maketitle
{\it Keywords. }
Saint-Venant equations, Green-Naghdi equations, shallow water,
weak solutions, nonlocal hyperbolic system, long waves, breakdown

{\it Mathematics Subject Classification.} 35B44, 35B60, 35Q35, 76B15, 35L67 






\section{Introduction}
The main aim of this paper is to demonstrate singularity formation 
for classical solutions of a system of 
regularized Saint-Venant (shallow-water) equations that was
introduced by Clamond and Dutykh in \cite{ClDu18}.
In conservation form in one space dimension, these regularized 
Saint-Venant (rSV) equations may be written
\begin{gather}
h_t + (hu)_x = 0 \,,
\label{e:rsvh}
\\ (hu)_t + (hu^2+\frac12 gh^2 + \epsl \calS )_x = 0,
\label{e:rsvm}
\\
 \calS \defeq
h^3 (-u_{tx}-u u_{xx}+u_x^2)
-gh^2 
\left(hh_{xx}+\frac12 h_x^2 \right)  \,.
\label{e:rsvR}
\end{gather}
Here $h(x,t)$ represents the depth of the fluid, 
$u(x,t)$ represents the average horizontal velocity of the fluid column,
$g$ is the gravitational constant, and $\epsl$ is a dimensionless
regularization parameter.
This system admits 
\blue{\emph{weakly singular shock layer} solutions} 
that were described in \cite{pu2018weakly}.

The rSV equations above were derived in \cite{ClDu18} 
as a non-dispersive variant of the Green-Naghdi equations 
\cite{GN74,green1976derivation} with zero surface tension 
(also called Serre equations \cite{serre1953b}).
Equations \eqref{e:rsvh}--\eqref{e:rsvR} follow from a least action
principle for the Lagrangian with density given by
\begin{equation}
\frac{\mathcal L}\rho = 
\frac12 hu^2 -\frac12 g h^2 
+ \epsl\left( 
\frac12 h^3 u_x^2 - \frac12 gh^2h_x^2
\right) + (h_t+(hu)_x)\phi \,,
\end{equation}
with a Lagrange multiplier field $\phi$ to enforce \eqref{e:rsvh}.
The Green-Naghdi equations with surface tension
take the same dimensional form as in \eqref{e:rsvh}-\eqref{e:rsvm},
but with 
$\epsl\calS$ above replaced by the quantity
\begin{equation}\label{d:R-GN}
\calS_{\rm GN} = \frac13 h^3 (-u_{tx}-u u_{xx}+u_x^2)
-\gamma \left(hh_{xx}-\frac12 h_x^2 \right)  \,,
\end{equation}
where $\gamma$ is the ratio of surface tension to density.
The Green-Naghdi equations
derive analogously from the Lagrangian with density
\begin{equation}
\frac{{\mathcal L}_{\rm GN}}\rho = 
\frac12 hu^2 -\frac12 g h^2 
+  
\frac16 h^3 u_x^2 - \frac12 \gamma h_x^2
+ (h_t+(hu)_x)\phi \,.
\end{equation}

The Green-Naghdi equations hold an important place among 
dispersive approximations to the full water wave equations, 
insofar as the small-slope assumptions they are based on are minimal
and they are capable of correctly approximating large-amplitude waves.
Many other dispersive water-wave models,
such as the Korteweg-de Vries, Camassa-Holm, and various Boussinesq systems,
can be derived from the Green-Naghdi equations by imposing further restrictions
on amplitude or structure; see the treatment by Lannes \cite{lannes2013water}.
Local-time well-posedness for the Green-Naghdi equations was studied in 
\cite{li2006shallow,ASLannes2008,israwi2011large}, 
and in \cite{li2006shallow} Y.~Li proved
that they constitute an approximation
to the water wave equations that is better than the classical shallow water
equations (which correspond to $\eps=0$ above).  
\blue{
And recently, the Green-Naghdi system has been found to have 
weakly singular peakon-like traveling-wave solutions, 
when the Bond number ${\rm Bo}=gh_\infty^2/\gamma$ takes the critical value 3
\cite{MDAZ2017,DHM2018}.
}
Yet, as far as we know, the \blue{analytical} question of whether smooth solutions
for the Green-Naghdi equations always exist globally in time,
or whether instead singularities may develop, remains open.

Smooth solutions of the rSV equations also satisfy a conservation law for energy, in the form
\begin{equation}\label{e:rsvE}
\eep_t + \qqp_x = 0 \,,
\end{equation}
{where}
\begin{align}
\label{d:eep}
&\eep \defeq \half hu^2  + \half gh^2 + \epsl \left(\half h^3u_x^2 +  \half gh^2h_x^2\right)\,,\\
&\qqp \defeq {\half h u^3 + gh^2 u} +  \epsl \left(\paren{\half h^3u_x^2 + \half gh^2h_x^2 +  
\calS}u
        +  gh^3h_xu_x\right)\,.
\end{align}
For $\eps>0$ and provided the fluid depth $h$ remains larger than a positive constant, 
this energy controls the $L^2$ norms of the derivatives of both $h$ and $u$, precluding
shock formation.
By comparison, the Green-Naghdi energy, given by 
\begin{equation}\label{e:EGN}
\eee_{\rm GN} \defeq \half hu^2  + \half gh^2 +  \frac16 h^3u_x^2 
+ \frac12\gamma h_x^2 \,,
\end{equation}
fails to control $h_x$ for $\gamma=0$, 
so one might guess the Green-Naghdi equations without surface tension
are ``less regularizing'' than the rSV equations.

When linearized about any constant state $(h_\star,u_\star)$, the opposite seems to be the case, however.
For the linearized rSV equations, the phase velocity of linear waves is independent of frequency. 
Thus the rSV equations are linearly {\em dispersionless}---they appear to lack a 
linear dispersive regularization mechanism.

This dispersionless nature of the rSV equations and the tendency of numerically computed solutions
not to generate oscillations and discontinuities were primary reasons given by
Clamond and Dutykh \cite{ClDu18} for interest in studying these equations. 
These authors pointed out, in fact, that 
the rSV equations are less accurate than the Green-Naghdi equations 
for approximating the exact water-wave dispersion relation 
(with zero surface tension) in the long-wave regime,
and only as accurate as the classical shallow-water system.
We note, however, that actually there is a physical regime where
the rSV equations 
do approximate the linear dispersion of water waves as accurately 
as Green-Naghdi equations.
From \eqref{e:rsvh}--\eqref{d:R-GN} above, clearly both systems
yield the same linearization at depth $h_\star$
when both $\epsl$ and the inverse Bond number 
${\rm Bo}\inv= \gamma/{g h_\star^2}$ take the value $\frac13$.
(It is well-known that linear dispersion vanishes in the Korteweg-de Vries 
approximation at this critical value of Bond number, and the same is clear from the 
dispersion relation for Green-Naghdi equations with surface tension
in \cite[Eq. (6.10)]{GN74}.)  
Even in this case, though, 
the nonlinear factor $gh^2$ in \eqref{e:rsvR} 
makes the rSV system formally less accurate than Green-Naghdi 
as a weakly nonlinear water-wave approximation,  
unless the amplitude variation is so small that the differences with
\eqref{d:R-GN} are of the same order as the terms neglected there.

What is more interesting for present purposes is the fact that the rSV 
equations admit a new kind of traveling wave, which is
a weakly singular analog of classical shallow-water shock waves:
It is shown in \cite{pu2018weakly} that for every such classical shock,
the rSV equations admit a corresponding non-oscillatory traveling 
wave solution which is continuous but only piecewise smooth, having a weak singularity at a single point 
where the energy is dissipated as it is for the classical shock.
Numerical evidence provided in \cite{pu2018weakly} suggests that a weak singularity can develop
from a smooth solution and start to dissipate energy after some positive time.

It is the purpose of this paper to partly address the question of well-posedness 
and whether singularity formation occurs for smooth solutions of the rSV equations.
Our goals are: 
(i) to provide a basic theory of local-time well-posedness and lifespan for classical solutions with sufficient Sobolev regularity;
(ii) to prove that depth $h$ remains strictly positive for small-energy perturbations of a constant state;
and (iii) to identify initial data for which no classical solution can exist globally in time.  
From our continuation criteria for solutions we infer that the sup norms of both $h_x$ and $u_x$ 
blow up as $t$ approaches the maximal time of existence.

{Our local-time well-posedness theorem (Theorem~\ref{thm:wellposedness} in section~3) 
deals with (possibly large) initial perturbations of a constant state in $H^s(\RR)$ for some 
\blue{real $s>\frac32$}, 
such that the depth $h$ is initially uniformly positive. 
The depth remains uniformly positive for small-energy perturbations---this
 follows from energy conservation and Proposition~\ref{p:Ebound} in section 2, 
which is essentially a Sobolev-type inequality.
\blue{In section~\ref{s:continuation} we establish criteria for finite-time blow-up
based on the sup norm of the derivatives $h_x$ and $u_x$ and/or the vanishing of $h$.
}
Our main blow-up argument (the proof of Theorem~\ref{thm:blowup} in section \ref{s:blowup}) 
identifies a class of small-energy initial data, 
defined by a few explicit inequalities (all listed in Lemma~\ref{lem:ICs}), 
for which $h_x$ and $u_x$ must both blow up in sup norm.
The nature of the blow up is that the derivative of one of the classical shallow-water Riemann invariants 
$R_\pm = u\pm 2\sqrt{gh}$ blows up to $-\infty$ while remaining bounded above, 
while the derivative of the other Riemann invariant remains bounded.
}

To give some insight into how our analysis will proceed,
observe that in the momentum equation \eqref{e:rsvm}, there are two terms involving time derivatives. 
It is natural to combine them and transform the momentum equation \eqref{e:rsvm} into a standard evolution equation 
for $u$. For a smooth function $w:\RR\to \RR$, define 
\begin{equation*}
\iii_h(w) = hw - \epsl (h^3 w_x)_x. 
\end{equation*}
or, in terms of composition of operators,
\begin{equation}  
\iii_h = h - \epsl\partial_x\circ h^3 \circ\partial_x.   \label{def:ih} 
\end{equation}
Formally acting $\iii_h^{-1}$ on both sides of the momentum equation \eqref{e:rsvm}, one obtains
\begin{equation}
u_t + gh_x + uu_x + \epsl \iii_h^{-1} \partial_x \paren{2h^3 u_x^2 - \half g h^2 h_x^2} = 0.  \label{e:rsvu}
\end{equation}
This is the standard evolution equation for horizontal velocity in the classical shallow-water system plus 
a nonlocal term. Because we expect the operator $\iii_h^{-1}$ gains two derivatives, the nonlocal term
is formally of {\em order zero} and represents a lower-order perturbation to the classical shallow-water system.

This is an important difference with the Green-Naghdi system 
as treated by Israwi in \cite{israwi2011large} without surface tension.
The system with constant surface tension $\gamma>0$ appears no better, 
\blue{even though the Green-Naghdi energy in \eqref{e:EGN} 
controls the $H^1$ norm of $h$ in that case. 
For, instead of \eqref{e:rsvu}, the momentum equation takes the form
\begin{equation*}
u_t + gh_x + uu_x + 
\iii_h^{-1} \partial_x \paren{
\frac23 h^3u_x^2 + (\frac13 g h^2 -\gamma) h h_{xx}+ \frac12\gamma h_x^2
} = 0.  
\end{equation*}
The trouble is that 
the nonlocal term is stronger here than in \eqref{e:rsvu},
remaining formally of order one in $h$, except when linearized at
the constant depth $h_\star$ where $\gamma = \frac13 gh_\star^2$, 
corresponding to Bond number ${\rm Bo}=3$. 
}

For the rSV system, then, equations \eqref{e:rsvh} and \eqref{e:rsvu} constitute
a nonlocal hyperbolic system for which
we are able to use a standard shallow-water symmetrizer to study well-posedness,
and study blowup using (coupled) Ricatti-type equations for the derivatives of classical Riemann invariants.
It turns out that, in addition to coupling the pair of Ricatti-type equations,
 the nonlocal terms contain a local part that alters the main quadratic terms.
This important contribution to the Ricatti-type equations 
appears to change the nature of blowup profiles as compared with the classical shallow-water case. We discuss this difference heuristically in section \ref{s:asym profile} below.

The rSV equations that we study in this paper also loosely resemble a number of 2-component
systems that generalize the Camassa-Holm equation; see \cite{CLZ2006,Ivanov2006,Kuz2007,HoNaTr2009,IonescuKruse2015}
for studies of such systems.  One of the more extensively studied systems of this kind,
appearing in \cite{CLZ2006,Ivanov2006,Kuz2007}, is an integrable 2-component Camassa-Holm system
that can be written in the form
\begin{gather}
  h_t + (hu)_x = 0, \\
  u_t + 3 u u_x - u_{txx} - 2u_xu_{xx} - uu_{xxx} + g hh_x = 0,
\end{gather}
In the context of shallow-water theory, this system has been derived by
Constantin and Ivanov \cite{CI2008} (see also \cite{IonescuKruse2013}).
For this system, derivative blow-up does not occur---smooth solutions exist globally in time
for all smooth initial data for which $h$ is initially strictly positive, see
\cite{CI2008,GuanYin2010,GuiLiu2010,GHR2012}.

{An interesting question that remains open
is whether the rSV equations admit globally defined weak solutions
for arbitrary initial perturbations small in $H^1(\RR)$. 
The rSV system does admit energy-conserving small-energy traveling waves 
with cusp singularities, as described in \cite{pu2018weakly}.
The scalar Camassa-Holm equation, 
which famously admits weak solutions that include peakon traveling waves,
has global existence for weak solutions that
may conserve the $H^1$ energy \cite{bressan2007conserved} 
or dissipate it \cite{bressan2007dissipate}.
An expected difference between the scalar Camassa-Holm equation
and the rSV system, however, is that in general we do not expect weak rSV 
solutions to conserve energy globally in time,
due to the presence of energy-dissipating weakly singular traveling waves.
}

\section{An energy criterion for uniform positivity of depth}

We begin the analysis of solutions of the rSV equations 
\eqref{e:rsvh}--\eqref{e:rsvR}
by establishing an explicit energy criterion that ensures the uniform positivity 
of the depth $h$ for small $H^1(\R)$ perturbations of any given 
constant state $(h_\star,u_\star)$ with $h_\star>0$, $u_\star\in\R$.
The proof resembles the proof of the Sobolev inequality for the $H^1$ norm,
and exploits the simple idea that for the surface to reach the bottom, 
relative energy has to be sufficiently large.
Our criterion has no apparent analog for the Green-Naghdi 
\blue{system with $\gamma=0$}
or the two-component Camassa-Holm system mentioned above,
because the energies for those systems do not control the integral of $h_x^2$.

Formally, a smooth solution $(h,u)$ of the rSV equations defined
for all $x\in\R$, with the property that $(h-h_\star,u-u_\star)\in H^1(\R)$ 
for all $t$, conserves the relative energy
\begin{equation}\label{d:Estar}
E_\star = \int_\R \half h(u-u_\star)^2 + \half g(h-h_\star)^2 
+ \half\epsl\paren{h^3 u_x^2 + gh^2 h_x^2}\,dx \,.
\end{equation}
By fixing $t$ and discarding the terms involving $u$, we infer the following.

\begin{proposition}\label{p:Ebound}
Let $h_\star>0$, $u_\star\in\R$ and suppose $(h-h_\star,u-u_\star)\in H^1(\R)$.
Then

(a) For all $x\in\R$ we have
\begin{equation}\label{e:Ehineq}
E_\star \geq \frac{g\sqrt\eps}3 (h(x)-h_\star)^2(2h(x)+h_\star).
\end{equation}

(b) If $E_\star< \frac13 g\sqrt\eps h_\star^3$, 
then we have $h(x)\geq h_E>0$ for all $x\in\R$, where $h_E$
is the unique solution in $(0,h_\star)$ of 
\begin{equation}\label{d:hE}
E_\star = \frac{g\sqrt\eps}3 (h_E-h_\star)^2(2h_E+h_\star).
\end{equation}
\end{proposition}

\begin{proof}
Because $\frac12(a^2+b^2)\geq \pm ab$, for any $x\in\R$ we have

\begin{align*}
E_\star &\geq {g\sqrt\eps} 
\left(
\int_{-\infty}^x (h-h_\star)h h_x \,dx
-\int_x^{\infty} (h-h_\star)h h_x \,dx\right)\\
&= \frac{g\sqrt\eps}3 (h(x)-h_\star)^2(2h(x)+h_\star).
\end{align*}
This proves (a). To deduce (b), note that the map 
$w\mapsto (w-h_\star)^2(2 w+h_\star)$ is strictly decreasing
for $w\in(0,h_\star)$.
\end{proof}

\begin{remark} \label{rem:hbounds}
(i) The lower bound $h(x)\geq h_E$ in part (b) is sharp, as one can see by 
choosing $h(x)$ to be an even function, determined on $[0,\infty)$ as the solution of 
\[
\sqrt\epsl h h_x = h_\star-h \,,\quad h(0)=h_E.
\]
\item[(ii)] For periodic functions on $\R$ having finite period $L$, the
same estimates hold with $E_\star$ obtained by integrating over a single period 
and with
$h_\star$ replaced by the average value of $h$ over one period.  One alters the proof
by replacing the endpoints $-\infty$ and $\infty$ by points $a$ and $a+L$ 
where $h(a)=h_\star$.
\item[(iii)] Using the upper bound in case (a), the lower bound in case (b) implies
that $(h(x)-h_\star)^2 \leq (h_E-h_\star)^2$, whence $h(x)\leq 2h_*-h_E$ for all $x$.
\item[(iv)] The part of the relative energy that we are using to bound the depth 
from below corresponds in principle to potential energy of the fluid. 
In an exact physical fluid model with zero surface tension, however, 
it is possible to perturb a flat fluid surface to reach the bottom 
with a small change in potential energy,
by creating a downward cusp on a tiny horizontal length scale. 
\end{remark}

\section{Local well-posedness, and scaling of lifespan}

In this section, we will establish finite-time 
\blue{existence and uniqueness} 
for solutions of the initial-value problem 
for the rSV system that have finite energy relative to a constant
state $(h_\star,0)$ with $h_\star>0$. (We take $u_\star=0$ without loss
due to Galilean invariance of the system.)
We will pay particular attention to how the existence time 
(lifespan of the solution) varies according to the value of 
the nonlinearity parameter $\alpha = a / h_\star$, 
where the parameter $a$ indicates the amplitude of the perturbation.
For example, in the inviscid Burgers equation $u_t+uu_x=0$, 
a Ricatti-type calculation for $u_x$ shows that the existence time 
for smooth solutions is proportional to $1/\alpha$.

For this reason, we make the following change of variables, writing
\begin{equation}\label{e:changevars}
h  = h_\star + \alpha\eta, \quad \text{ and replacing $u$ by }\alpha u.
\end{equation}
Here and below we retain the notation $h=h_\star+\alpha\eta$ for brevity, however.
The scaled pair $(\eta,u)$ now satisfies the following system:
\begin{gather}
\eta_t +  (hu )_x = 0  \,,
\label{e:rsvh_ul}
\\ h(u_t + g\eta_x +  \alpha uu_x) + \epsl\alpha 
\tilde{\calS}_x = 0\,,
\label{e:rsvm_ul}
\\
 \alpha\tilde{\calS} =
h^3 (- u_{tx} - \alpha u u_{xx}+\alpha u_x^2 )- gh^2\left( h\eta_{xx}+\frac12  \alpha \eta_x^2 \right)  \,.
\label{e:rsvR_ul}
\end{gather}
In terms of $\iii_h = h - \epsl\partial_x\circ h^3 \circ\partial_x$, 
we observe that we can reformulate the momentum equation \eqref{e:rsvm_ul} as 
\begin{equation}
u_t + g\eta_x + \alpha uu_x + \epsl\alpha \iii_h^{-1}\partial_x\paren{2h^3 u_x^2 - \half g h^2\eta_x^2} = 0. \label{e:rsvu_ul}
\end{equation}

Equations \eqref{e:rsvh_ul}, \eqref{e:rsvu_ul} form a (nonlocal) hyperbolic system
that takes the form
\begin{equation}  \label{e:system}
W_t + B(W)W_x + F(W) = 0,
\end{equation}
with $W = (\eta, u)^T$ and where
\begin{equation}
B(W) = \paren{\begin{array}{ll}
\alpha u & h \\ g & \alpha u
\end{array}}, \quad
F(W) = \paren{\begin{array}{c}
0 \\ f(W)
\end{array}} ,
\end{equation}
with 
\begin{equation}
f(W) = \epsl\alpha \iii_h^{-1}\partial_x\paren{2h^3 u_x^2 -\half g h^2\eta_x^2}.
\label{d:fW}
\end{equation}
For this system, we shall use a standard iteration scheme for symmetrizable hyperbolic systems
to prove the main theorem of this section. 
We remark that both of the parameters $\alpha$ and $\epsl$ are dimensionless,
and there is some interest in understanding how solutions 
behave in the regime when one or both parameters become small.

\begin{theorem}  \label{thm:wellposedness}
Fix $h_\star>0$.  Let \blue{$s> \frac32$ be real}, and let $\epsl,\alpha\in(0,1]$. 
Assume the initial data $W^0 = (\eta^0, u^0)^T \in H^s(\RR)$ and satisfies 
\begin{equation}
h^0_{\min}\defeq \inf_{x\in\RR} (h_\star + \alpha \eta^0(x)) >0.
\label{nonzero depth condition}
\end{equation}
Then there exists $T_0 = T_0(s,\|W^0\|_{H^s},h^0_{\min})>0$
independent of $\epsl$ and $\alpha$,
such that the regularized shallow-water system \eqref{e:system} 
admits a unique solution 
\[
W = (\eta,u)^T\in 
C([0, T_0/\alpha]; H^s(\RR))
\,\cap\, C^1([0, T_0/\alpha]; H^{s-1}(\RR))\,,
\]
having the initial condition $W^0$ 
and preserving the positive depth condition 
\[
\inf_{x\in\R}h(x,t)>0.
\]
Moreover, the following conservation of energy property holds:
$\tilde{E}_\star  = $const, where
\begin{equation}
\tilde E_\star = \half \int_\R  hu^2 +  g\eta^2
+ \epsl\paren{h^3 u_x^2 + gh^2 \eta_x^2}\,dx \,.
\label{e:tildeE}
\end{equation}
\end{theorem}

\begin{remark}
(i) In Theorem~\ref{thm:wellposedness}, 
the dependence on $h^0_{\min}$ can be dropped if the 
initial relative energy $E_\star=\alpha^2\tilde E_\star$ 
is so small that Proposition~\ref{p:Ebound}(b) applies.

\item[(ii)] 
Although continuous dependence on the initial data is not mentioned in the theorem, we do have it in the following sense:
for all initial data $\tilde{W}^0$ satisfying 
$\norm{\tilde{W}^0}_{H^s} \leq 2 \norm{W^0}_{H^s}$ 
with uniformly positive depth $\tilde h\geq \hsstar>0$,
the corresponding solution $\tilde W$ satisfies
\begin{equation}
\norm{\tilde{W} - W}_{L^\infty([0,T]; H^{s-1})} \leq C(s, \hsstar, \norm{W}_{L^\infty([0,T]; H^s)}) \norm{\tilde{W}^0 - W^0}_{H^s}\,,
\label{e:approxW}
\end{equation}
on any common time interval of existence where $\tilde h, h\geq \hsstar>0$.
The proof follows in a standard way analogous to the convergence
proof of the iteration scheme for existence; see \cite{majda1984,pu_thesis}
for details.

\item[(iii)] 
 \blue{When $s\ge2$,} the relative energy satisfies the
following conservation law in a strong $L^2$ sense:
\begin{equation}\label{e:rsvE_ul}
\tilde{\eee}^\epsl_t + \tilde{\qqq}^\epsl_x = 0 \,,
\end{equation}
{with}
\begin{align}
&\tilde{\eee}^\epsl \defeq \half hu^2  + \half g\eta^2 + \epsl  \left(\half h^3u_x^2 +  \half gh^2\eta_x^2\right)\,,\\
&\tilde{\qqq}^\epsl \defeq {\half \alpha h u^3 + g\eta h u} +  
\epsl \left(\paren{\half  h^3u_x^2 + \half  g h^2\eta_x^2 +  \tilde{\calS} }\alpha u
		+  gh^3\eta_xu_x\right)\,,
\end{align}
where 
we find using \eqref{e:rsvu_ul} that $\tilde{\calS}$ from 
\eqref{e:rsvR_ul} satisfies
\begin{equation}
\tilde{\calS} = \paren{I + \epsl h^3 \D_x \iii_h\inv\D_x}
\paren{2h^3 u_x^2 - \frac12 g h^2\eta_x^2}\,.
\end{equation}
(For $s\ge2$ this expression will belong to $H^1(\RR)$.)
\end{remark}

The proof of Theorem~\ref{thm:wellposedness} is structured as follows:
Subsection \ref{subsec:preliminaries} contains preliminary estimates,
including technical analysis of the operator $\iii_h$.
Subsection \ref{subsec:linear analysis} analyzes the iteration step in the iteration scheme and establishes the needed {\em a priori}
energy estimates. 
The main proof of Theorem \ref{thm:wellposedness} is presented in subsection \ref{subsec:large time exist.}.

\subsection{Preliminary results} \label{subsec:preliminaries}
The elliptic operator $\iii_h$ plays an important role in the energy estimate 
and well-posedness of the regularized shallow-water system. In this subsection,
we shall introduce the main technical tools to handle $\iii_h$ and the nonlocal
term in \eqref{e:rsvu_ul}.

Let $D=\D_x$ and let $\Lambda=(I-\D_x^2)^{1/2}$ be the operator associated 
with Fourier symbol $(1+\xi^2)^{1/2}$, 
so that $\widehat{\Lambda u} = (1+\xi^2)^{1/2} \hat{u}$ 
{for all tempered distributions }$u$.
\blue{ 
We will make use of two well-known harmonic analysis results 
which we cite here without proofs.
The first one is a Kato-Ponce commutator estimate \cite{KatoPonce1988},
\begin{equation}\label{kato ponce2}
\norm{[\Lambda^s, \phi]\psi}_{L^2} \leq  
C(s)\paren{ \norm{D \phi}_{L^\infty} \norm{\Lambda^{s-1}\psi}_{L^2}+ \norm{\Lambda^s \phi}_{L^2} \norm{\psi}_{L^\infty} }.
\end{equation}
valid for all $\phi\in H^s(\RR)$, $D\phi\in L^\infty(\RR)$ 
and $\psi\in H^{s-1}(\RR)\cap L^\infty(\RR)$,
for all real $s\ge0$.
The second one is a classical ``tame'' product estimate 
(also proved in \cite{KatoPonce1988}),
\begin{equation}\label{moser tame2}
\norm{\Lambda^s (\phi\psi)}_{L^2} \leq 
C(s)\paren{ \norm{\phi}_{L^\infty}\norm{\Lambda^s \psi}_{L^2} + \norm{\Lambda^s \phi}_{L^2}\norm{\psi}_{L^\infty}  }\,.
\end{equation}
valid for all $\phi,\psi \in H^s(\RR)\cap L^\infty(\RR)$
and all real $s\ge0$.
}

\blue{
The following lemma establishes the invertibility of $\iii_h$ 
and bounds $\iii_h^{-1}\partial_x$.
It improves the bounds on $\iii_h^{-1}\partial_x$ claimed in  
Lemma~2 of \cite{israwi2011large} in two ways, 
bounding one derivative more and providing a tame estimate
that is needed for the blow-up analysis in section~\ref{s:blowup}.
} 

\begin{lemma}   \label{lemma:ih-1} 
Let $\hsstar>0$ and $\epsl\in(0,1]$ and suppose
$h\in W^{1,\infty}(\RR)$ satisfies
\begin{equation}\label{nonzero depth condition2}
h(x)\geq \hsstar \quad\mbox{ for all $x\in\R$.} 
\end{equation}
Then: 
\begin{itemize}
\item[1)] The operator
$\iii_h = h - \epsl\partial_x\circ h^3 \circ\partial_x$ 
from $H^2(\RR)$ to $L^2(\RR)$ is an isomorphism.
\blue{ 
\item[2)] Let $s\ge0$ and $h - h_\star\in H^{s}(\RR)$.
Then for any $\psi\in H^{s}(\RR)$, 
the function 
\[
w = \eps\iii_h^{-1}\partial_x \psi 
\]
belongs to $H^{1+s}(\R)$ and satisfies the estimate
\begin{align}
\norm{w}_{H^{1+s}} \leq\  &  
\hat C_1 \Big(\norm{\psi}_{H^{s}} + 
 \norm{h-h_\star}_{H^{s}} 
 ( \eps^{-1/2}\norm{w}_{L^\infty} + \norm{w_x}_{L^\infty} ) \Big)\,,
\label{elliptic estimate2}
\end{align}
where 
$\hat C_1 = C(s, \hsstar, \norm{h-h_\star}_{W^{1,\infty}})$ 
independent of $\psi$, $\eps$ and $\alpha$.
\item[3)] If furthermore $s>\frac12$, then
\begin{equation} 
\label{elliptic estimate1}
\norm{w}_{H^{1+s}} \leq  
\hat C_2
\,\norm{\psi}_{H^s} (1+\norm{h-h_\star}_{H^s}) \,
\end{equation}
where 
$\hat C_2 = C(s, \hsstar, \norm{h-h_\star}_{W^{1,\infty}})$
independent of $\psi$, $\eps$ and $\alpha$.
} 
\end{itemize}
\end{lemma}

\begin{remark}
The estimate \eqref{elliptic estimate2} will be improved below in
Lemma~\ref{lem:wwx bounds}, to provide bounds on $\|w\|_{L^\infty}$ and
$\norm{w_x}_{L^\infty}$ in terms of $\psi$ that will be used to
prove a blow-up criterion.
\end{remark}

\begin{proof}
1. The idea is that $\iii_h$ is in essence a very well-behaved elliptic operator such that the basic Lax-Milgram approach works on it.

We define the bilinear mapping $a: H^1(\RR)\times H^1(\RR) \to \RR$ such that
\begin{equation}
a(u,v) = (hu, v)_{L^2} + \epsl(h^3u_x, v_x)_{L^2}\quad \forall\, u,v \in H^1(\RR).
\end{equation}

Next, we will show that $a$ is not only bounded but also coercive. We have
\begin{align*}
\abs{a(u,v)} &\leq \norm{h}_{L^\infty} \norm{u}_{L^2}\norm{v}_{L^2} 
		+ \epsl \norm{h}_{L^\infty}^3 \norm{u}_{H^1}\norm{v}_{H^1}\nonumber\\
&\leq C(h)\norm{u}_{H^1}\norm{v}_{H^1}.\nonumber
\end{align*}
and by \eqref{nonzero depth condition2}
\begin{equation}
\abs{a(u,u)} \geq \hsstar \norm{u}_{L^2}^2
		+ \epsl (\hsstar)^3 \norm{u_x}_{L^2}^2
		\geq \epsl C(\hsstar)\norm{u}_{H^1}^2.
\end{equation}
So by Lax-Milgram, there is a bounded bijective linear operator $\tilde I: H^1(\RR)\to H^{-1}(\RR)$
such that 
\begin{equation}
a(u,v) = \angl{\tilde Iu, v}_{H^{-1}\times H^1} \quad\forall\, u,v \in H^1(\RR).
\end{equation}
Therefore, given any $f\in L^2(\RR)\hookrightarrow H^{-1}(\RR)$, $u := \tilde I^{-1}f$ satisfies 
\begin{eqnarray}
(f,v)_{L^2} = \angl{f, v}_{H^{-1}\times H^1} = a(u,v) 
\quad\forall\, v \in H^1(\RR).
\end{eqnarray}
It follows that the distributional derivative 
$(\eps h^3 u_x)_x=f-hu\in L^2(\RR)$, whence $h^3u_x\in H^1(\RR)$.
Hence $u\in H^2(\RR)$ and $ \iii_h u = f$. It follows $\iii_h$ is
an isomorphism.

2. Let $\norm{\cdot}_{H_\eps^1}$ be the norm on $H^1(\RR)$ (equivalent to $\norm{\cdot}_{H^1}$ but not uniformly in $\epsl$) 
determined by
\begin{equation}
\norm{u}_{H_\eps^1}^2 := \norm{u}_{L^2}^2 + \epsl \norm{u_x}_{L^2}^2.
\end{equation}
Consider $\phi, \psi\in C_c^\infty(\RR)$ and $u\in H^2(\R)$ such that
\begin{equation}
\iii_h u = \phi + \sqrt{\epsl} \psi_x \,.
\end{equation}
Invoking the coercivity estimate from above, we obtain
\begin{eqnarray}
C(\hsstar)\norm{u}_{H_\eps^1}^2 &\leq & a(u,u) = (\iii_h u, u)_{L^2} = (\phi,u)_{L^2} - (\psi, \sqrt{\epsl}u_x)_{L^2}   \nonumber\\
&\leq & (\norm{\phi}_{L^2} + \norm{\psi}_{L^2}) \norm{u}_{H_\eps^1}\,,
\end{eqnarray}
hence
\begin{equation}\label{lemma: base case}
\norm{\sqrt\eps u}_{H^1}\le
\norm{u}_{H_\eps^1}  \leq  C(\hsstar) (\norm{\phi}_{L^2} + \norm{\psi}_{L^2}).
\end{equation}
Choosing $\phi=0$, $w=\sqrt\eps u$,
this proves the case $s = 0$ 
for \eqref{elliptic estimate2}.

\blue{ 
Next, assume $\phi=0$ and note
\[
hu - \eps h^3 u_{xx} = \eps (h^3)_x u_x + \sqrt{\epsl}\psi_x \,.
\]
Test this against $-u_{xx}$ and integrate by parts. We obtain
\begin{align*}
a(u_x,u_x) &= -(u_x,h_x u)
-(\sqrt\epsl u_{xx}, (h^3)_x \sqrt\epsl u_{x} + \psi_x )
\\ & \le \frac12\hsstar \norm{u_x}_{L^2}^2 
+ \frac12 \hsstar^3 \norm{\sqrt\epsl u_{xx}}_{L^2}^2  
+ \hsstar^{-3}\norm{\psi_x}_{L^2}^2
\\ & \quad 
+  C(\hsstar,\norm{h-h_\star}_{W^{1,\infty}}) \norm{u}_{H_\eps^1}^2
\end{align*}
By \eqref{lemma:  base case} we can then infer that
\begin{equation}
\norm{\sqrt\eps u_x}_{H^1} \le \norm{u_x}_{H_\eps^1} \le 
C(\hsstar,\norm{h-h_\star}_{W^{1,\infty}})
\norm{\psi}_{H_1} \,.
\label{w:H2est}
\end{equation}
By interpolation, it follows that for every $s\in[0,1]$, 
\begin{equation}
\norm{\sqrt\epsl u}_{H^{s+1}} + \norm{u}_{H^s} \le 
C(\hsstar,\norm{h-h_\star}_{W^{1,\infty}})
\norm{\psi}_{H^s} \,,
\label{uHsx estimate}
\end{equation}
} 

\blue{
3. Next, for any $s>0$, noting $\Lambda^s\iii_h u = \sqrt\eps\D_x\Lambda^s\psi$,
we compute 
\begin{equation}
\iii_h(\Lambda^s u) = [h,\Lambda^s]u - \epsl\D_x[h^3,\Lambda^s]u_x + \sqrt\eps\D_x\Lambda^s\psi \,,
\end{equation}
so using \eqref{lemma: base case} with $\phi$ and $\psi$ replaced by 
\begin{equation}
\tilde{\phi} = [h, \Lambda^s]u, \quad \tilde{\psi} = \Lambda^s \psi - \sqrt{\epsl} [h^3, \Lambda^s]u_x
\end{equation}
we find, after using the Kato-Ponce commutator estimate~\eqref{kato ponce2}, that 
\begin{align}
\norm{\Lambda^s u}_{H_\eps^1}  
&\leq   C(\hsstar)
  \paren{ \norm{[h, \Lambda^s]u}_{L^2} + \norm{ \Lambda^s \psi - \sqrt{\epsl} [h^3, \Lambda^s]u_x}_{L^2}  } \nonumber\\
&\leq C(s, \hsstar) \Big(\norm{h_x}_{L^\infty}\norm{u}_{H^{s-1}} + \norm{h-h_\star}_{H^s}\norm{u}_{L^\infty} +
	 \norm{\psi}_{H^s}     \nonumber\\
&	+ \norm{ (h^3)_x}_{L^\infty}\norm{\sqrt{\epsl}u_x}_{H^{s-1}} 
	+ \norm{h^3 - h_\star^3}_{H^s}\norm{\sqrt{\epsl}u_x}_{L^\infty}  \Big).
\end{align}
After Taylor-expanding $h^3-h_\star^3$ in powers of $h-h_\star$ 
and using the tame product estimate \eqref{moser tame2}, we infer 
\begin{align*}
\norm{\Lambda^s u}_{H_\eps^1}  
&\leq \hat C_1
\paren{  \norm{\Lambda^{s-1} u}_{H_\eps^1} +\norm{\psi}_{H^s} 
		+ \norm{h-h_\star}_{H^s} ( \norm{u}_{L^\infty} + \norm{\sqrt{\epsl}u_x}_{L^\infty} ) },
\end{align*}
where $\hat C_1$ is a generic constant depending upon
$s$, $\hsstar$, and $\norm{h-h_\star}_{W^{1,\infty}}$,
independent of $\psi$, $\eps$, $\alpha$.
Note that for $s\le1$, $\norm{\Lambda^{s-1}u}_{H_\eps^1}\le \norm{u}_{H_\eps^1}$.
Hence by induction starting from \eqref{lemma:  base case}, we deduce that
for all $s\ge0$,
\begin{equation}
  \norm{\Lambda^s u}_{H_\eps^1} \leq 
  \hat C_1
  \Big(\norm{\psi}_{H^s} 
  + \norm{h-h_\star}_{H^s} ( \norm{u}_{L^\infty} + \norm{\sqrt{\epsl}u_x}_{L^\infty} ) \Big).
\label{Lsu estimate}
\end{equation}
With $w=\sqrt\eps u$ as before, 
since $\norm{w}_{H^{s+1}}\le \|\Lambda^s u\|_{H_\eps^1}$ we deduce
that \eqref{elliptic estimate2} holds.
} 

\blue{
4. Finally, if $s>\frac12$ then due to the embedding $H^s\hookrightarrow L^\infty$,
from \eqref{uHsx estimate} we infer 
\[
\norm{u}_{L^\infty} + \norm{\sqrt{\epsl}u_x}_{L^\infty}  \le
C(s,\hsstar,\norm{h-h_\star}_{W^{1,\infty}})
\norm{\psi}_{H^s} \,.
\]
Using this together with \eqref{Lsu estimate} proves \eqref{elliptic estimate1},
and concludes the proof. 
 } 
\end{proof}

\subsection{Linear analysis} \label{subsec:linear analysis}
The local-time existence of solutions to the system \eqref{e:system} is proved by 
a standard approach for symmetrizable hyperbolic systems, 
based on proving convergence of the following iteration scheme: 
Set $W_0(x,t) = W^0(x)$ and inductively
determine $W=W_{n+1}$ from $\underline{W}=W_n$ for $n\geq0$ 
by solving the (linear) initial value problem  
with coefficients and source term frozen at 
the (now given) reference state $\underline{W}  \in C([0, T/\alpha]; H^s)$:
\begin{equation}
\left\{\begin{array}{l}
W_t + B(\underline{W})W_x + F(\underline{W}) = 0\,,\\
W|_{t = 0} = W_0 \,.
\end{array}
\right.  \label{e:ivp}
\end{equation}

This subsection is devoted to the proof of energy estimates for this linear initial value problem.
A symmetrizer for $B(\underline{W})$ is given by
\begin{equation}
A(\underline{W}) = \paren{\begin{array}{ll}
g & 0 \\ 0 & \underline{h}
\end{array}}.
\end{equation}
(Here $\underline{h} = h_\star+\alpha\underline{\eta}$ where $\underline{W} = (\underline{\eta},\underline{u})^T$.)
A natural energy for the IVP \eqref{e:ivp} is
\begin{equation}
E^s(W) \defeq (\Lambda^s W, A(\underline{W})\Lambda^s W) =  
g\norm{\eta}^2_{H^s} + (\Lambda^s u, \underline{h}\,\Lambda^s u) .
\label{d:Es}
\end{equation}
which is equivalent to $\norm{W}_{H^s}^2$ provided that $0 < \hsstar \leq \abs{\underline{h}} \leq \norm{\underline{h}}_{L^\infty} < \infty$.

The following theorem establishes that the iteration scheme is well-defined,
and provides an energy estimate that controls the norms of 
all the solutions in the scheme.

\begin{theorem}[energy estimate] \label{thm:energy estimate}
 Fix $h_\star>0$. Let $s>\frac32$, 
$\hsstar\in(0,h_\star)$ and $R>0$.
Then there exists constants $T$, $K>0$ depending upon $s$, $\hsstar$ and $R$
but independent of $\epsl,\alpha\in(0,1]$, with the following property.  
Assuming that $W_0 = (\eta_0,u_0)\in H^s$ satisfies
 \begin{equation}
 \label{i:hEs0}
  h_0 > 2 \hsstar \quad\text{and}\quad  E^s(W_0) < \frac R2,
 \end{equation}
\blue{%
and that $\underline{W}=(\underline{\eta},\underline{u})
\in C([0,T/\alpha];H^s)\cap C^1([0,T/\alpha];H^{s-1})$ satisfies
 \begin{equation}
 \label{i:hEsu}
  \underline{h} \geq \hsstar \,,\quad
  E^s(\underline{W})\le R \,, \quad
  \norm{\underline{W}_t}_{H^{s-1}} \le K
 \quad\text{for all $t\in[0,T/\alpha]$, }
 \end{equation}
 there exists a unique solution $W=(\eta,u)^T \in C([0,T/\alpha];H^s)$ to \eqref{e:ivp}
 satisfying 
 \begin{equation}
 \label{i:hEs}
  {h} \geq \hsstar \,,\quad
 E^s({W})\le R\,,\quad
  \norm{W_t}_{H^{s-1}} \le K
 \quad\text{for all $t\in[0,T/\alpha]$, }
  \end{equation}
}
and furthermore
\begin{equation}
\label{i:expEs}
E^s(W(\cdot,t)) \leq e^{C\alpha t}E^s(W_0) +  e^{C\alpha t} - 1 < R 
\end{equation}
 for all $t\in[0,T/\alpha]$, for some $C = C(s, \hsstar, R) > 0$.
\end{theorem}

\begin{proof}
Since all coefficients of the initial value problem \eqref{e:ivp} are independent of unknowns, 
by a standard Friedrichs mollification approach we have the well-posedness of the
symmetrizable hyperbolic system. We will focus on the proof of the (a priori) energy estimate.

For simplicity, we use underlines to denote the dependence on $\underline{W}$:
$$\underline{A}:= A(\underline{W}), \quad\underline{B}:= B(\underline{W}),
\quad \underline{F}:= F(\underline{W}),\quad \underline{f}:= f(\underline{W}).
$$
We compute that
\begin{eqnarray}
\partial_t {E^s(W)} &=& \partial_t (\Lambda^s W, \underline{A}\Lambda^s W)  \nonumber\\
&=& (\underline{h}_t\Lambda^s u, \Lambda^s u) + 2(\Lambda^s W_t, \underline{A}\Lambda^s W)
\label{e:dtEs}
\end{eqnarray}
Using equation \eqref{e:ivp} and integrating by parts, we obtain
\begin{eqnarray}
\partial_t {E^s(W)} &=& (\underline{h}_t\Lambda^s u, \Lambda^s u) 
		- 2(\underline{A}\underline{B}\Lambda^s W_x, \Lambda^s W ) \nonumber\\
&&+ 2([\underline{B}, \Lambda^s] W_x, \underline{A} \Lambda^s W) 
		- 2 (\Lambda^s\underline{F}, \underline{A} \Lambda^s W). \label{e:ES1}
\end{eqnarray}
Now we turn to bound each of the four terms on the right-hand side of \eqref{e:ES1} in turn.
In the estimates below, various constants denoted by $C$ may
change from line to line without changing the notation.

\blue{
1) 
Since $\norm{\underline{h}_t}_{L^\infty} = 
\alpha\|\underline{\eta}_t\|_{L^\infty} \le\alpha C(s)K$
due to the embedding $H^{s-1}\hookrightarrow L^{\infty}$,
\begin{equation}
\abs{ (\underline{h}_t\Lambda^s u, \Lambda^s u)} \leq \norm{\underline{h}_t}_{L^\infty}\norm{u}_{H^s}^2
		\leq \alpha C(s, \hsstar)K E^s(W) \,.
\end{equation}
}
2) For the second term, note that
\begin{equation}
\underline{A}\underline{B} = \paren{\begin{array}{cc}
\alpha g\underline{u} & g\underline{h} \\ g\underline{h} & \alpha\underline{hu} 
\end{array}}
\end{equation}
is symmetric, so we can take advantage of this symmetry and 
move the derivative from $W$
terms to the $\underline{AB}$ term. We use
\begin{equation}
\abs{\underline{h}_x} \leq \alpha C(s)\norm{\underline{\eta}}_{H^s} \leq \alpha C(s,\hsstar) E^s(\underline{W})^{1/2} \leq \alpha C(s,\hsstar,R)
\end{equation}
together with the analogous bound $|\underline{u}_x|\le C(s,\hsstar,R)$
and obtain
\begin{eqnarray}
&& 2\abs{(\underline{A}\underline{B}\Lambda^s W_x, \Lambda^s W )} 
= \abs{((\underline{AB})_x\Lambda^s W, \Lambda^s W )}  \nonumber \\
&&\qquad \leq  \abs{(\alpha g \underline{u}_x \Lambda^s \eta, \Lambda^s \eta)}  + \abs{2( g\underline{h}_x \Lambda^s \eta, \Lambda^s u)}
		+ \abs{(\alpha (\underline{hu})_x \Lambda^s u, \Lambda^s u)} \nonumber\\
&&\qquad \leq   \alpha C(s, \hsstar, R) E^s(W).
\end{eqnarray}

3) For the third term, it is crucial to use the 
Kato-Ponce commutator estimate \eqref{kato ponce2}
together with the embedding 
$H^{s-1} \hookrightarrow L^\infty$: 
\begin{eqnarray}
\abs{([\underline{B}, \Lambda^s] W_x, \underline{A} \Lambda^s W) }   
&= & \abs{ \paren{ \alpha[\underline{u}, \Lambda^s]\eta_x + \alpha [\underline{\eta},\Lambda^s]u_x, g\Lambda^s\eta}
		+  \paren{ \alpha [\underline{u}, \Lambda^s]u_x, \underline{h}\Lambda^s u } }   \nonumber\\
&\leq & \alpha C(s)\paren{ \norm{ \underline{u} }_{H^s}\norm{\eta_x}_{H^{s-1}} 
  + \norm{ \underline{\eta} }_{H^s}\norm{u_x}_{H^{s-1}}  }  \norm{\eta}_{H^s}  \nonumber\\
&& +\ \alpha C(s)  \norm{ \underline{u} }_{H^s} \norm{u_x}_{H^{s-1}} \norm{\underline{h}}_{L^\infty} \norm{u}_{H^s}  \nonumber\\
&\leq & \alpha C(s,\hsstar,R) \paren{\norm{\eta}_{H^s}^2  + \norm{\eta}_{H^s}\norm{u}_{H^s} 
  \blue{ + \norm{u}_{H^s}^2}} \nonumber \\
&\leq &  \alpha C(s, \hsstar, R) E^s(W).
\end{eqnarray}

4) For the fourth term (the nonlocal term), we exploit part 3) of Lemma~\ref{lemma:ih-1}
(replacing $1+s$ by $s$ and using the bound $\norm{\underline{h}-h_\star}_{W^{1,\infty}}\le C R$)
to get 
\begin{eqnarray}
\norm{\underline{f}}_{H^s} &=& 
\norm{\epsl\alpha \iii_{\underline{h}}^{-1}\partial_x(2\underline{h}^3 \underline{u}_x^2 -\half g\underline{h}^2\underline{\eta}_x^2)}_{H^s} \nonumber\\
&\leq & \alpha  \blue{C(s,\hsstar, R)}
\norm{2\underline{h}^3 \underline{u}_x^2 -\half g\underline{h}^2\underline{\eta}_x^2}_{H^{s-1}}
\leq  \alpha C(s, \hsstar, R)\,,
\end{eqnarray}
where the last inequality is obtained by expanding $\underline{h} = h_\star+\alpha \underline{\eta}$ and 
using the fact that $H^{s-1}(\R)$ is a Banach algebra.
Then we deduce
\begin{eqnarray}
\abs{(\Lambda^s\underline{F}, \underline{A} \Lambda^s W)}
&=& \abs{\paren{\Lambda^s \underline{f}, \underline{h}\Lambda^s u }}
\leq \blue{\norm{\underline{h}}_{L^\infty}}
 \norm{\underline{f}}_{H^s} \norm{u}_{H^s}    \nonumber\\
&\leq& \alpha C(s, \hsstar, R)(1+ E^s(W)).
\end{eqnarray}
\blue{
5) Before proceeding further, we bound $W_t$ using \eqref{e:ivp}, obtaining
\begin{eqnarray}
\norm{W_t}_{H^{s-1}} &=& \norm{B(\underline{W})W_x+\underline{F}}_{H^{s-1}} 
\nonumber\\
&\le& C(s,\hsstar,R)(E^s(W)+ 1) \,. 
\label{e:Wtbound}
\end{eqnarray}
We fix the choice of $K=K(s,\hsstar,R)$ at this point, requiring that
\[
C(s,\hsstar,R)(R+1)<K\,.
\]
6) Now, substituting all estimates back into \eqref{e:ES1}, we find
the Gronwall-type differential inequality
\begin{equation}\label{e:ES2}
\partial_t E^s(W) \leq \alpha 
C(s, \hsstar, R) \paren{E^s(W) +  1 }.
\end{equation}
This in turn gives the energy inequality
\begin{equation}\label{e:ES3}
E^s(W) \leq 
e^{C\alpha t}E^s(W_0) + e^{C\alpha t} - 1. 
\end{equation}
We choose $T=T(s,\hsstar,R)>0$ small enough so that
\[
e^{CT}E^s(W_0) + e^{CT} - 1< R. 
\]
7) Now pointwise, we have the bound
\[
|h_t| = \alpha |\eta_t| = \alpha |\underline{u} h_x + \underline{h} u_x| \le 
\alpha C(s,\hsstar)(E^s(\underline{W})+ E^s(W))
\]
hence from the hypotheses on the initial data,
\begin{equation}
h(x,t) = h_0(x) + \int_0^t h_t(x,\tau)\,d\tau 
> 2\hsstar - \alpha t C(s,\hsstar)(R+E^s(W)).
\label{e:hbound}
\end{equation}
Making $T$ smaller if necessary, we can ensure that
\begin{equation}
 \hsstar > C(s,\hsstar) 2RT \,.
\end{equation}
Now, considering \eqref{e:ES3}, \eqref{e:Wtbound}, and \eqref{e:hbound} in turn,
we conclude that the inequalities in \eqref{i:hEs} and \eqref{i:expEs} all hold as desired.
} 
\end{proof}

\bigskip
\subsection{Proof of Theorem \ref{thm:wellposedness}} \label{subsec:large time exist.}
The rSV system has a structure highly resembling that of the classical 
shallow-water system.  With the energy estimate, proofs of 
existence and uniqueness are standard, so we omit details.

{The proof that the relative energy $\tilde E_\star$ is conserved 
relies on a few basic facts: \blue{Provided} $s\geq2$, we have 
\[
\eta,u\in C([0,T],H^2(\RR))\cap C^1([0,T],H^1(\RR)).
\]
Also, for any $v,w\in H^1(\RR)$,
$\int_{\RR} v w_x =-\int_{\RR}v_x w$ and $\iii_h\inv (v_x)\in H^2(\RR)$, with
\[
\int_{\RR} (h u - \epsl(h^3 u_x)_x)\iii_h\inv (v_x) = \int_{\RR} u v_x.
\]
Using these facts, the details of checking that $\D_t\tilde E_\star=0$ 
from \eqref{e:system} are rather tedious but straightforward, so we omit them.
\blue{For general initial data $W_0\in H^s(\R)$ with $s>\frac32$, we can approximate by initial data 
$\tilde W_0\in H^2(\R)$ and infer from \eqref{e:approxW} that relative
energy $\tilde E_\star$ is constant in time for the $H^s$ solution $W$.}
}

\section{Blow-up Criteria}
\label{s:continuation}

\blue{In this section, our aim is to establish the following criteria for finite-time
breakdown of regular solutions.}
\begin{theorem}\label{thm:continuation}
Let $[0,T_{\max})$ be the maximal interval of existence of the solution from 
Theorem~\ref{thm:wellposedness}.
Then if $T_{\max} < \infty$ we have either
\begin{equation}\label{e:Wxblowup}
\limsup_{t\to T_{\max}^-}\norm{\partial_x W(\cdot,t)}_{L^\infty} = \infty \,,
\end{equation}
or
\begin{equation}\label{e:0depth}
\liminf_{t\to T_{\max}^-} \,
\inf_{x\in\RR} h(x,t) = 0 \,.
\end{equation}
\end{theorem}

From the uniform positivity criterion in Proposition~\ref{p:Ebound} and  
the change of variables in \eqref{e:changevars} 
(which implies $\alpha^2\tilde E_\star=E_\star$), we deduce
that for small-energy perturbations of constant states,
a finite maximal existence time implies derivative blow-up.
\begin{corollary} \label{cor:blowup}
If $\alpha^2\tilde E_\star <\frac13 g\sqrt\epsl h_\star^3$
and the maximal time of existence $T_{\max}<\infty$, 
then  
\begin{equation}\label{e:blowup2}
\limsup_{t\to T_{\max}^-}\norm{\partial_x W(\cdot,t)}_{L^\infty} = \infty \,.
\end{equation}
\end{corollary} 

This result follows from the fact that under the given hypothesis,
\begin{equation}
\inf_{x\in\R} h(x,t) \ge h_E>0
\end{equation}
for all $t\in[0,T_{\max})$, 
with $h_E$ given by Proposition~\ref{p:Ebound}, so 
\eqref{e:0depth} cannot occur.

\blue{
We will use a Gronwall-type inequality to derive the blow-up criterion
of Theorem~\ref{thm:continuation}. 
For this argument, it is necessary to improve the estimate on the nonlocal
operator $\iii_h\inv$ from \eqref{elliptic estimate2} 
by controlling the $L^\infty$ norms that appear explicitly on the right-hand side.
Toward this aim, the following classical Landau-Kolmogorov interpolation inequality is crucial.
}
\begin{lemma}
Let $\phi\in C^2(\RR; \RR)$ be such that
\begin{equation*}
\norm{\phi}_{L^\infty} < \infty, \quad \norm{\phi''}_{L^\infty} < \infty.
\end{equation*}
Then
\begin{equation}
\norm{\phi'}^2_{L^\infty} \leq 2 \norm{\phi}_{L^\infty}\norm{\phi''}_{L^\infty}. \label{landau kol}
\end{equation}
\end{lemma}

\begin{proof}
Without loss we may assume $\norm{\phi''}_{L^\infty} = 1$ and $\phi'(0) = a > 0$. Then
\begin{eqnarray}
\phi(a) &=& \phi(0) + \int_0^a \phi'(x) \,dx = \phi(0) + \int_0^a \paren{\phi'(0) + \int_0^x \phi''(y) \,dy} \,dx  \nonumber\\
&\geq & \phi(0) + \int_0^a \paren{a + \int_0^x (-1) \,dy} \,dx = \phi(0) + \half a^2
\end{eqnarray}
Similarly, $\phi(-a) \leq \phi(0) - \half a^2$. So
\begin{equation}
2\norm{\phi}_{L^\infty} \geq \phi(a) - \phi(-a) \geq a^2 = \abs{\phi'(0)}^2.
\qquad\qedhere
\end{equation}
\end{proof}

\blue{
The nonlocal operator $\iii_h\inv$ will be bounded using the following lemma.
Below, the Banach space of continuous functions $\phi\colon\R\to\R$ 
having finite limits $\phi(\pm\infty)$ at $\pm\infty$ is denoted
by $\Clim(\R)$, or just $\Clim$. We let $C_0$ denote the subspace of continuous 
functions on $\R$ that vanish at $\pm\infty$, having limits $\phi(\pm\infty)=0$.
\begin{lemma} \label{lemma:ih-1No2}
Let $s > \frac32$ and let $h - h_\star \in H^s(\RR)$ 
with $0 < h_{\min} \leq h \leq h_{\max} < \infty$. 
Then 
$\iii_h^{-1}$ is well-defined on $\Clim$. 
If $\phi\in \Clim$ then
$h\phi\in\Clim$, and 
with $v = \iii_h^{-1}(h\phi)$, it holds that $v\in \Clim$ 
and $v_x$, $v_{xx}\in C_0$, with 
\begin{equation}
\norm{v}_{L^\infty} \leq  \norm{\phi}_{L^\infty}\quad \text{and }\quad
 \norm{v_x}_{L^\infty} \leq \frac{2}{\sqrt{\epsl}}\frac{h_{\max}^2}{h_{\min}^3}\norm{\phi}_{L^\infty}.
\label{e:vests}
\end{equation}
\end{lemma}

\begin{proof}
1.  We first show that the $L^\infty$ estimates hold for $\phi\in L^2(\RR)\cap C_0$.
In this case, $v\in H^2(\RR)$, since $\iii_h^{-1}: L^2(\RR) \to H^2(\RR)$ 
is well-defined by Lemma \ref{lemma:ih-1}. Moreover, $v$ is $C^2$ since $h$ is $C^1$
and $hv - \eps (h^3 v_{x})_x=h\phi$.

We introduce a new variable $z$ on $\RR$ such that
\begin{equation*}
\frac{d}{dz} = h^3 \frac{d}{dx}.  
\end{equation*}
Then in terms of the new variable we have
\begin{equation}
v - \epsl h^{-4} v_{zz} = \phi.   
\end{equation}
By the maximum principle it follows
$\norm{v}_{L^\infty} \leq  \norm{\phi}_{L^\infty}$. 
Therefore
\begin{equation}
\epsl \norm{v_{zz}}_{L^\infty} = \norm{h^4(v - \phi)} _{L^\infty}
        \leq  {2h_{\max}^4} \norm{\phi}_{L^\infty}   ,   
\end{equation}
hence the Landau-Kolmogorov interpolation inequality \eqref{landau kol} implies
\begin{equation*}
(h_{\min})^3\norm{v_x}_{L^\infty} \leq \norm{v_z}_{L^\infty}
        \leq \paren{2\norm{v}_{L^\infty} \norm{v_{zz}}_{L^\infty}}^{1/2}    
\leq   \frac{2 h_{\max}^2}{\sqrt{\epsl}} \norm{\phi}_{L^\infty}.
\end{equation*}
Since $L^2(\R)\cap C_0$ is dense in $C_0$, it follows $\iii_h\inv$ is well-defined
on $C_0$, and the estimates \eqref{e:vests} hold for all $\phi\in C_0$.

2. Now consider an arbitrary $\phi\in \Clim$, and define 
$w_\phi \in C^\infty(\RR)$ by
\begin{equation}
w_\phi (x) = \frac{\phi(-\infty)}{h_\star} + \frac{\phi(+\infty) - \phi(-\infty) }{h_\star}\frac{e^x}{1+e^x} \,.
\end{equation}
This function has the property that all its derivatives vanish at $\pm\infty$, and 
\begin{equation}
hw_\phi(-\infty) = \phi(-\infty), \quad hw_\phi(+\infty) = \phi(+\infty). 
\end{equation}
Then $\phi - \iii_h w_\phi \in C_0$. 
Due to step 1, we may then define 
\begin{equation}
u= \iii_h^{-1}\phi \defeq w_\phi + \iii_h^{-1}( \phi - \iii_h w_\phi  ).
\end{equation}
Note that  $u$ is $C^2$ and
\begin{equation}
\iii_h u = \phi \,.
\end{equation}
The estimates \eqref{e:vests} on $v=\iii_h\inv(h\phi)$ follow similarly as in step 1.
\end{proof}

Next, we observe that $\Dx\iii_h^{-1}\Dx\circ h^3$ is a 
nonlocal operator of order zero. 
We can extract the local part of this operator, as follows.
Since $\iii_h = h - \epsl \partial_x\circ h^3 \circ \partial_x$,
\begin{equation}
-\epsl\Dx \iii_h^{-1}\Dx \circ h^3 \circ\Dx = \Dx - \Dx\iii_h^{-1}\circ h. \label{e:nonlocal decomp2}
\end{equation}
This motivates us to write, for any nice enough function $\phi$, 
\begin{equation}\label{e:nonlocal decomp3}
 - \epsl \partial_x\iii_h^{-1}\partial_x (h^3 \phi) = \phi - \partial_x \iii_h^{-1}\paren{h\int_{-\infty}^x \phi } \,.
\end{equation}
In light of this, we have the following $L^\infty$ estimates,
which in particular improve the estimate \eqref{elliptic estimate2}
in part 2) of Lemma~\ref{lemma:ih-1}.

\begin{lemma}\label{lem:wwx bounds}
Let $s>\frac32$ and suppose $h-h_\star\in H^s$ with $0 < h_{\min} \leq h \leq h_{\max} <\infty$.
\begin{itemize}
\item[1)] If $w = \eps\iii_h^{-1}\D_x\psi$ with 
$\psi\in L^1\cap \Clim$, then
\begin{equation}
\norm{w}_{L^\infty} + \norm{w_x}_{L^\infty} \leq C(\epsl, h_{\min}, h_{\max}) \paren{\norm{\psi}_{L^\infty} + \norm{\psi}_{L^1}  } \,.
\end{equation}

\item[2)] If furthermore $\psi\in H^{s-1}$, then 
\begin{equation}
\norm{w}_{H^s}\le \hat C_3\left(
\norm{\psi}_{H^{s-1}} + \norm{h-h_\star}_{H^{s-1}}
(\norm{\psi}_{L^\infty} +\norm{\psi}_{L^1})
\right) \,,
\label{e:wpsitame}
\end{equation}
where $\hat C_3 = C(s,\eps,h_{\min},\norm{h-h_\star}_{W^{1,\infty}})$.
\end{itemize}
\end{lemma}

\begin{proof}
From \eqref{e:nonlocal decomp3}, we have
\begin{eqnarray}
 - \epsl \partial_x\iii_h^{-1}\partial_x \psi = h^{-3}\psi - \partial_x \iii_h^{-1}\paren{h\int_{-\infty}^x h^{-3}\psi } \,,
\end{eqnarray}
Due to Lemma~\ref{lemma:ih-1No2},
\begin{eqnarray}
\norm{\partial_x \iii_h^{-1}\paren{h\int_{-\infty}^x h^{-3}\psi }}_{L^\infty}
&\leq & C(\epsl, h_{\min}, h_{\max}) \norm{\int_{-\infty}^x h^{-3}\psi }_{L^\infty} \nonumber \\ 
&\leq & C( \epsl, h_{\min}, h_{\max} ) \norm{\psi}_{L^1}.
\end{eqnarray}
Then it follows
\begin{equation} 
\norm{w_x}_{L^\infty} = 
\norm{\eps\partial_x\iii_h^{-1}\partial_x \psi}_{L^\infty}  
\leq  C(\epsl, h_{\min}, h_{\max}) 
\paren{ \norm{\psi}_{L^\infty} + \norm{ \psi }_{L^1}  }.
\end{equation}
From definition of $\iii_h^{-1}$, we also have
\begin{equation}
-\epsl \iii_h^{-1} \partial_x \circ h^3\circ \partial_x = \operatorname{Id} - \iii_h^{-1}\circ h
\end{equation}
whence, again due to Lemma~\ref{lemma:ih-1No2},
\begin{eqnarray}
\norm{w}_{L^\infty} &=& \norm{\eps \iii_h^{-1}\partial_x \psi}_{L^\infty}   
	= \norm{\eps\iii_h^{-1}\partial_x \circ h^3\circ \partial_x \paren{\int_{-\infty}^x h^{-3}\psi} }_{L^\infty}  \nonumber\\
& \leq  &C(\epsl, h_{\min}, h_{\max})\norm{ \psi }_{L^1}.
\end{eqnarray}
This proves part 1). To deduce part 2), simply use the result of part 1)
together with part 2) of Lemma~\ref{lemma:ih-1}.
\end{proof}

Next we apply these results to provide a tame estimate on the nonlocal
term in the system~\eqref{e:system}.

\begin{corollary} \label{cor:fWbound}
Let $s>\frac32$, let $W = (\eta,u)\in H^s$, 
and suppose $h=h_\star+\alpha\eta$ satisfies $0 < h_{\min} \leq h\leq h_{\max}$.
Let $E=\tilde E_\star$ be the energy given by \eqref{e:tildeE}.
Then with $f(W)$ given by \eqref{d:fW}, we have 
\begin{equation}\label{ineq:tame estimate2}
\norm{ f(W)  }_{H^s}  
\leq   \hat C
\norm{W}_{H^s}\,, \quad\mbox{where}\quad
\hat C =C(s,\eps,h_{\min},E,\norm{W}_{W^{1,\infty}}).
\end{equation}
\end{corollary}

\begin{proof}
Let $\psi = 2h^3u_x^2 - \half gh^2\eta_x^2$, 
so $w = \eps\iii_h^{-1}\D_x\psi=f(W)$. 
Then clearly
\begin{equation}
\norm{\psi}_{L^\infty} \leq C(\norm{W}_{W^{1,\infty}}) 
\quad \text{and}\quad 
\norm{\psi}_{L^1} \leq  C(h_{\min}, h_{\max}) E \,.
\end{equation}
Combined with \eqref{e:wpsitame}, this implies
\begin{equation}
\norm{ f(W)}_{H^s}
\leq  \hat C 
\Big(\norm{\psi}_{H^{s-1} } + \norm{\eta}_{H^{s-1} } 
\Big).
\label{e:wpsilast-1}
\end{equation}
where $\hat C=C(s,\epsl, h_{\min},E, \norm{W}_{W^{1,\infty}})$.
By applying the tame product estimate \eqref{moser tame2} 
several times to the terms of $\psi$ 
(expanding $h$ in powers of $h-h_\star$), we deduce
\[
\norm{\psi}_{H^{s-1} } \leq C(s,\norm{W}_{W^{1,\infty}}) \norm{W}_{H^s}\,.
\]
Combining this with \eqref{e:wpsilast-1}, we obtain \eqref{ineq:tame estimate2}.
\end{proof}

Now we are ready to present the proof of the blow-up criterion.
} 

\begin{proof}[Proof of Theorem~\ref{thm:continuation}]
Suppose $W\in C([0,T_{\max});H^s)$ is the solution of \eqref{e:system} 
from Theorem~\ref{thm:wellposedness} on
a maximal time interval with $T_{\max}<\infty$.  We claim that it is 
impossible that 
\begin{equation} \label{c:continuation}
\sup_{t\in[0,T_{\max})} \norm{W_x(\cdot,t)}_{L^\infty} <\infty
\quad\text{and}\quad
h_{\min}= \inf_{t\in[0,T_{\max})} \inf_\RR h(\cdot,t) >0.
\end{equation}
Suppose on the contrary that \eqref{c:continuation} holds. Then
because the energy $\tilde E_\star$ is conserved,
we have 
\begin{equation}
\norm{W(\cdot,t)}_{L^\infty} \leq C \norm{W(\cdot,t)}_{H^1} 
\leq C(h_{\min}, \epsl)\norm{W^0}_{H^1},
\end{equation}
hence $\norm{W(\cdot,t)}_{W^{1,\infty}}$ remains bounded on $[0,T_{\max})$.
We claim that $\norm{W(\cdot,t)}_{H^s}$ also remains bounded. 
By part (i) of Theorem~\ref{thm:wellposedness} it then follows 
we can continue the solution to a larger time interval, 
contradicting maximality of $T_{\max}$.

\blue{
To bound $\norm{W(\cdot,t)}_{H^s}$ we modify the previous energy estimates
as follows. 
We define a new energy by
\begin{equation}
\tilde E^s (W) \defeq 
(\Lambda^s W, A(W)\Lambda^sW) = g\norm{\eta}_{H^s}^2+ 
(\Lambda^s u, h\Lambda^s u),
\end{equation}
replacing $\underline{W}$ in \eqref{d:Es} by $W$.
Due to the positive depth condition in \eqref{c:continuation} we have
\begin{equation}
\tilde E^s(W) \geq C(s,h_{\min}) \norm{W}_{H^s}^2 \,.
\label{i:tEs}
\end{equation}
Similarly to \eqref{e:dtEs} but with $\underline{W}=W$ we compute
\begin{eqnarray}
\Dt (\Lambda^s W,A \Lambda^s W) &=& (h_t  \Lambda^s W,\Lambda^s W) + ((AB)_x \Lambda^s W,\Lambda^s W)
\nonumber\\&& 
+\ 2([B,\Lambda^s ]W_x,A\Lambda^s W) - 2(\Lambda^s  F,A\Lambda^s W).
\label{e:dtEstilde}
\end{eqnarray}
We now revise the previous estimates of the four terms as follows.

\begin{itemize}
\item[1)] For the first term,
$\abs{h_t} = \alpha \abs{(hu)_x} \leq C(\alpha,\norm{W}_{W^{1, \infty}})$,
so
\begin{equation}
\abs{(h_t\Lambda^s  u, \Lambda^s  u)} \leq C\norm{W}_{H^s}^2.
\end{equation}

\item[2)] For the second term,
\begin{eqnarray}
&&\abs{((AB)_x\Lambda^s  W_x, \Lambda^s  W )}    \nonumber\\
&&\quad \leq  \abs{(\alpha g u_x \Lambda^s  \eta, \Lambda^s  \eta)} + \abs{2(g h_x \Lambda^s  \eta, \Lambda^s  u)}
		+\abs{(\alpha (hu)_x \Lambda^s  u, \Lambda^s  u)}   \nonumber\\
&&\quad \leq   C(s,\norm{W}_{W^{1,\infty}} )\norm{W}_{H^s}^2.
\end{eqnarray}

\item[3)] For the third term, using the Kato-Ponce commutator estimate
\eqref{kato ponce2} and the tame product estimate \eqref{moser tame2},
we get
\begin{eqnarray}
&&\abs{([B, \Lambda^s ] W_x, A \Lambda^s  W) }   \nonumber\\
&&\quad \leq  \abs{(\alpha [u, \Lambda^s ]\eta_x, g \Lambda^s  \eta)} + \abs{(\alpha [\eta, \Lambda^s  ]u_x, g\Lambda^s  \eta)} 
		+ \abs{(\alpha [u, \Lambda^s ]u_x, h\Lambda^s  u)}   \nonumber\\
&&\quad \leq   C(s,\norm{W}_{W^{1,\infty}} )\norm{W}_{H^s}^2.
\end{eqnarray}

\item[4)] For the fourth term, the estimate in Corollary~\ref{cor:fWbound}
yields
\begin{equation}
\norm{ f(W)  }_{H^s}  \leq   
C(s,\eps,h_{\min},E,\norm{W}_{W^{1,\infty}}) \norm{W}_{H^s}\,, 
\end{equation}
where $E=\tilde E_\star(W^0)$ is the constant energy of the solution.
Hence
\begin{equation}
\abs{(\Lambda^s  F, A \Lambda^s  W)} = \abs{\Lambda^s  f, h\Lambda^s  u} 
		\leq 
C(s,\epsl, \norm{W}_{W^{1,\infty}})
\norm{W}_{H^s}^2\,.
\end{equation}
\end{itemize}
Collecting everything yields
\begin{equation}
\partial_t {\tilde E^s(W)} \leq    C \norm{W}_{H^s}^2 \leq C \tilde E^s(W) \,,
\end{equation}
due to \eqref{i:tEs}. By Gronwall's inequality it follows $\tilde E^s(W(\cdot,t))$
remains bounded on $[0,T_{\max})$, hence the same is true for $\norm{W(\cdot,t)}_{H^s}$.

This completes the proof of Theorem~\ref{thm:continuation}.
} 
\end{proof}

\section{Derivative blow-up in finite time} \label{s:blowup}
In this section, the main goal is to show that solutions 
to \eqref{e:rsvh}, \eqref{e:rsvu} 
with certain initial conditions do exhibit derivative blow-up.
The general strategy is to show that derivatives of the classical
Riemann invariants satisfy coupled Ricatti-type equations 
that must exhibit blow-up.

\subsection{Ricatti-type equations for derivatives of Riemann invariants}
Write the Riemann invariants $R_\pm$ of the classical shallow water system and the two corresponding characteristic speeds $\lambda_\pm$ as
\begin{eqnarray}
R_+ = u + 2\sqrt{gh},&& \lambda_+ = u + \sqrt{gh},    \nonumber\\
R_- = u - 2\sqrt{gh}, && \lambda_- = u - \sqrt{gh}.    \label{def: R lambda}
\end{eqnarray}
These quantities satisfy
\begin{equation}
\lambda_+ = \frac{1}{4}(3R_+ + R_-), \quad \lambda_- = \frac{1}{4}(R_+ + 3R_-).
\label{e:lambdapmR}
\end{equation}
Next, note that
the function inside the nonlocal term in \eqref{e:rsvu} takes the form
\begin{eqnarray}
2h^3u_x^2 - \half g h^2 h_x^2 = 2h^3(u_x^2 - \frac{1}{4}g h^{-1}h_x^2) = 2h^3 (\lambda_+)_x (\lambda_-)_x\,.
\end{eqnarray}
From this one finds that
the evolution equations for $R_\pm$ along characteristic curves take the form
\begin{eqnarray}
\frac{d^+}{dt} R_+ := (R_+)_t + \lambda_+ (R_+)_x = -2\epsl \iii_h^{-1}\partial_x\paren{h^3 (\lambda_+)_x (\lambda_-)_x},
\label{e:r plus} \\
\frac{d^-}{dt} R_- := (R_-)_t + \lambda_- (R_-)_x = -2\epsl \iii_h^{-1}\partial_x\paren{h^3 (\lambda_+)_x (\lambda_-)_x}.
\label{e:r minus}
\end{eqnarray}
Here $\frac{d^+}{dt}, \frac{d^-}{dt}$ indicate the derivatives along ``$+$'' and ``$-$'' characteristic curves, respectively.

Next, we derive evolution equations for the derivatives of these classical
Riemann invariants, writing
\begin{equation}
P_+ = (R_+)_x = u_x + \sqrt{\frac gh}h_x \,, \quad 
P_- = (R_-)_x = u_x - \sqrt{\frac gh}h_x \,.
\label{d:Ppm}
\end{equation}
Clearly
\begin{equation}\label{e:lambdax}
(\lambda_+)_x = \frac{1}{4}(3P_+ + P_-), \quad (\lambda_-)_x = \frac{1}{4}(P_+ + 3P_-)\,.
\end{equation}
Differentiating \eqref{e:r plus}--\eqref{e:r minus}, one obtains that $P_+$ and $P_-$ satisfy a system of Riccati-type equations containing 
a nonlocal term:
\begin{eqnarray}
\frac{d^+}{dt} P_+
        = - \frac{1}{4}(3P_+ + P_-)P_+
 - 2\epsl \partial_x\iii_h^{-1}\partial_x\paren{h^3 (\lambda_+)_x (\lambda_-)_x},   \label{e1:Pp}\\
\frac{d^-}{dt} P_-
        = - \frac{1}{4}(P_+ + 3P_-)P_- - 2\epsl \partial_x\iii_h^{-1}\partial_x\paren{h^3 (\lambda_+)_x (\lambda_-)_x}.   
\label{e1:Pm}
\end{eqnarray}

\blue{
In this system, the nonlocal operator $\Dx\iii_h^{-1}\Dx\circ h^3$ has a 
local part which we extract as in the previous section,
using the formula \eqref{e:nonlocal decomp3}.
}
This motivates us to introduce a primitive for the product 
$(\lambda_+)_x(\lambda_-)_x$, writing
\begin{equation}\label{d:G}
G(y,t) = \int_{-\infty}^y 
(\lambda_+)_x(\lambda_-)_x\,dx.
\end{equation}
In terms of this quantity we can write 
\begin{equation}\label{e:nonlocal decomp}
 - 2\epsl \partial_x\iii_h^{-1}\partial_x\paren{h^3 (\lambda_+)_x (\lambda_-)_x} = 
2(\lambda_+)_x(\lambda_-)_x - Q \,,
\end{equation}
where
\begin{equation} \label{d:Q}
  Q \defeq 2 \Dx\iii_h^{-1}(hG)\,.
\end{equation}
Using \eqref{e:lambdax} we find that the Ricatti-type evolution equations for
$P_\pm$ take the form
\begin{eqnarray}
\frac{d^+}{dt} P_+
&=& -\frac{3}{8}P_+^2 + P_+ P_- + \frac{3}{8}P_-^2  - Q \,, 
        \label{e:p plus}\\
\frac{d^-}{dt} P_-
        &=& \phantom{-}\frac{3}{8}P_+^2 + P_+ P_- - \frac{3}{8}P_-^2  - Q\,. 
        \label{e:p minus}
\end{eqnarray}
These two equations are of central importance because the nonlocal term $Q$ that appears here is 
essentially a constant, and after that \eqref{e:p plus}--\eqref{e:p minus} is a system whose behaviors are 
governed by the quadratic terms in $P_+$ and $P_-$.

To see that the integral in \eqref{d:G} is well-defined, note that
\begin{align}
|(\lambda_+)_x(\lambda_-)_x| 
&= \frac1{16}|(3P_+ + P_-)(P_+ + 3P_-)| 
\nonumber\\ & 
\leq \frac12 (P_+^2 + P_-^2)
\leq 2\left( u_x^2 + \frac gh h_x^2 \right).
\label{i:lambdapm}
\end{align}
It then follows from conservation of the relative energy $E_\star$ in 
\eqref{d:Estar} that as long as $h\ge h_{\min}>0$ we have the estimate
\begin{equation}\label{i:Gbound}
\norm{G(\cdot,t)}_{L^\infty} \leq \int_\R
|(\lambda_+)_x(\lambda_-)_x| \,dx \leq \frac2{h_{\min}^3} 
\int_\R
\paren{h^3 u_x^2 + gh^2 h_x^2}\,dx
\leq \frac {4 E_\star}{\epsl h_{\min}^3} \,.
\end{equation}

To handle the nonlocal term $Q$ in the Ricatti-type system \eqref{e:p plus}--\eqref{e:p minus}, 
we exploit Lemma \ref{lemma:ih-1No2} to address the $L^\infty$ bound on the solution of the elliptic operator 
$\iii_h$.
Lemma \ref{lemma:ih-1No2} together with the estimate in \eqref{i:Gbound} immediately yields
the following uniform bound for the nonlocal term $Q$ given by \eqref{d:Q}.
\begin{proposition} \label{p:Qbound} 
For any classical solution $W$ of \eqref{e:system} satisfying 
$0<h_{\min}\leq h\leq h_{\max}<\infty$ on a time interval $[0,T_\star)$,  we have
\begin{equation}\label{b:Q}
\|Q(\cdot,t)\|_{L^\infty} \leq  \frac{4}{\sqrt\epsl} \frac{h_{\max}^2}{h_{\min}^3}\|G(\cdot,t)\|_{L^\infty} 
\leq \frac{16}{\epsl^{3/2}} \frac{h_{\max}^2}{h_{\min}^6} E_\star.
\end{equation}
\end{proposition}

Now we are ready to state our main theorem.
\begin{theorem}  \label{thm:blowup}
Fix $\eps,\alpha>0$.
Then there exists compactly supported smooth initial data $W^0=(\eta^0,u^0)$ 
of the IVP \eqref{e:system}, having arbitrarily small relative energy $\tilde E_\star$, 
for which the derivatives of the solution 
will blow up in finite time. 
The precise meaning of this is that there exists $T\in(0,\infty)$ 
such that the solution exists and stays smooth for all $(x,t)\in \RR\times [0,T)$, and
\begin{equation}
\sup_{\RR\times[0,T)} P_+(x,t)+|P_-(x,t)| <\infty,
\end{equation}
but
\begin{equation}
\inf_{x\in\RR} 
P_+(x,t) \to -\infty \quad\text{as $t\uparrow T$}.
\label{lim:Ppblowup}
\end{equation}
\end{theorem}
\begin{remark}
i) The blow-up behavior described 
in this theorem implies that 
\[
\inf_{x\in\RR} u_x(x,t) \to -\infty
\quad\text{and}\quad
\inf_{x\in\RR} 
h_x(x,t) \to -\infty 
\qquad\text{as $t\uparrow T$}, 
\]
see \eqref{e:uxhx} and \eqref{b:h} below.

ii) We will show that blow-up as described in the theorem occurs
for any initial data that satisfy certain explicit upper bounds
on relative energy $E_\star$, $|P_-|$, and $P_+$, 
such that $\inf P_+(\cdot,0)$ is sufficiently negative; 
see Lemma~\ref{lem:ICs} below.
\end{remark}
%

\subsection{Proof of derivative blow-up}
Next, we will sketch some of the fundamental ideas of the main proof.
Our goal is to construct initial data such that 
$P_+$ blows up while $P_-$ stays bounded. 
If indeed $P_-$ stays bounded then 
it is rather easy to infer from \eqref{e:p plus} 
$P_+$ blows up quickly if $P_+$ is initially large on some individual characteristic.
However, to show $P_-$ stays bounded everywhere while $P_+$ blows up somewhere,
\eqref{e:p minus} requires us to show that
the integral of $P_+^2$ along all the ``-'' characteristics has to remain bounded. 
A principal difficulty is that the characteristic speeds 
depend (nonlocally) on the solution itself.
Moreover, due to \eqref{e:lambdax} we expect both $(\lambda_+)_x$ and $(\lambda_-)_x$ 
to blow up to $-\infty$, as $P_+$ does. 
This indicates that characteristics curves are concentrating in the vicinity of the singularity. 

Let us introduce the flow maps $X_+, X_-:\RR\times [0, \infty) \to \RR$ along the ``$+$'' and ``$-$'' characteristic curves,
defined through
\begin{equation}
\left\{\begin{array}{ll}
\displaystyle  \frac{\partial X_+}{\partial t}(\xi, t) = \lambda_+(X_+(\xi, t), t)\\[8pt]  X_+(\xi, 0) = \xi
\end{array}
\right., \quad
\left\{\begin{array}{ll}
\displaystyle \frac{\partial X_-}{\partial t}(\zeta, t) = \lambda_-(X_-(\zeta, t), t)\\[8pt]  X_-(\zeta, 0) = \zeta
\end{array}
\right.,    \label{def: flow map}
\end{equation}
where $\xi, \zeta$ are Lagrangian variables. Differentiating the first set of equations in $\xi$, one obtains
\begin{equation}
\paren{\frac{\partial X_+}{\partial\xi}}_t = (\lambda_+)_x \frac{\partial X_+}{\partial\xi}, 
        \quad \frac{\partial X_+}{\partial\xi}(\xi, 0) = 1.     \label{e:plus char.}
\end{equation}
So, if some constant $L\geq (\lambda_+)_x$ everywhere along a certain ``+'' characteristic curve
for time in $[0,t]$, it follows that $\frac{\partial X_+}{\partial\xi} \leq e^{Lt}$ everywhere on the curve
in this time interval.  The same holds true for the ``-'' characteristic curves.
When it happens that $L\leq0$, 
nearby characteristics curves focus towards each other and concentrate as time increases.

The key to the proof of blow-up will be to establish two things:
\begin{itemize}
\item[(a)] $\frac{\partial X_+}{\partial \xi}P_+^2$ is close to constant along the ``+'' characteristic curves; 
i.e. the concentrating effect 
of the ``+'' characteristic curves and the blow-up effect of $P_+^2$ offset and exactly balance each other.
\item[(b)] The integrals of $P_+^2$ along the ``-'' characteristic curves are bounded.
\end{itemize}
We shall use (a) to derive (b). The exact meaning of (a) rests on the fact that,
\blue{temporarily fixing $\xi$ and abusing notation 
to write $P_+=P_+(X_+(\xi,t),t)$,
\begin{eqnarray}
&&\frac{d}{dt}\paren{\frac{\partial X_+}{\partial \xi}P_+^2} = 
\paren{\frac{\D}{\D t} \frac{\partial X_+}{\partial \xi}}P_+^2 
+ \frac{\partial X_+}{\partial \xi}2P_+\frac{d}{dt}(P_+)  \nonumber\\
&&\quad = (\lambda_+)_x \frac{\partial X_+}{\partial \xi}P_+^2 
        + \frac{\partial X_+}{\partial \xi}2P_+ \paren{ \half (\lambda_+)_x (3 P_--P_+) - Q} \nonumber\\
&&\quad= \frac{\partial X_+}{\partial \xi}P_+\paren{3P_-(\lambda_+)_x  - 2 Q} \,.
\label{e:plus derivative 0}
\end{eqnarray}
}Here the important point is that the highest order terms (cubic in $P_+$) match exactly and go away.

\begin{proof}[Proof of Theorem~\ref{thm:blowup}] 
We are now ready to begin the main argument, proceeding in several steps.

\textbf{Step Zero}.
We will work with smooth initial data with relative energy sufficiently small
so that we can apply Proposition~\ref{p:Ebound}(b) and Corollary~\ref{cor:blowup}. 
We introduce several explicit constants in this proof chosen as follows: 
The initial depth at infinity $h_\star > 0$ is arbitrary. 
We define positive constants $C_1(\epsl), C_2(\epsl, h_\star), C_3(\epsl, h_\star)$ explicitly by
\begin{equation}\label{d:C123}
C_1(\epsl) \defeq  \frac{6g}{\sqrt{\epsl}},  \quad
 C_2(\epsl, h_\star) \defeq \frac{72C_1}{\sqrt{\epsl}h_\star} , \quad C_3 (\epsl, h_\star) \defeq \frac{16C_1}{\sqrt{2gh_\star}}.
\end{equation}

We let $T_{\max}$ denote the maximal time of existence of the smooth solution,
and establish some preliminary bounds for 
solutions whose relative energy from \eqref{d:Estar} satisfies
\begin{equation}
E_\star \leq \frac{1}{6}g\sqrt{\epsl}h_\star^3.
\label{initial energy}
\end{equation}
Using this bound in Proposition~\ref{p:Ebound}(b) we get $h_E\ge\frac12h_\star$, and from this
and Remark~\ref{rem:hbounds}(iii) it follows that up to the maximal time of existence 
the depth satisfies the bounds
\begin{equation} \label{b:h}
h_{\min} \leq h \leq h_{\max}\,,
\quad\text{with}\quad  
h_{\min}=\frac12 h_\star\,, \quad h_{\max}=\frac32 h_\star\,. 
\end{equation}
Next, we can bound the fluid velocity by a Sobolev-like inequality,
writing
\begin{align}
\|u\|_{L^\infty}^2 
&\leq \int_\R 2|uu_x|\,dx \leq 
\int_\R (h u^2 + \eps h^3 u_x^2) \frac{dx}{\sqrt\epsl h^2}
\leq \frac{2 E_\star}{\sqrt\epsl h_{\min}^2} \leq   gh_{\max}.
\label{b:u}
\end{align}
Hence the characteristic speeds from \eqref{def:  R lambda} are bounded by 
\begin{equation}\label{b:lambdapm}
\|\lambda_\pm\|_{L^\infty} \leq 2\sqrt{g h_{\max}}
\end{equation}
Next, as in \eqref{i:lambdapm}--\eqref{i:Gbound},
we find that
\begin{equation}\label{b:Ppm}
\norm{P_\pm(\cdot, t)}_{L^2}^2 \leq 
\frac{4 E_\star}{\epsl h_{\min}^3} < C_1.
\end{equation}
Finally, applying Proposition~\ref{p:Qbound}, 
we deduce that 
\begin{equation}\label{b:Q2}
\|Q(\cdot,t)\|_{L^\infty}  
\leq  \frac{4 h_{\max}^2}{\sqrt\epsl h_{\min}^3} C_1 = \frac{72 C_1}{\sqrt\epsl h_\star} \leq C_2.
\end{equation}

\textbf{Step One}.
The key to the proof will involve obtaining bounds on the quantity
\begin{equation}\label{d:Mt}
M(t) \defeq \sup_{x\in\RR,s\in[0,t]} P_+(x,s)+ 
\sup_{x\in\RR,s\in[0,t]} |P_-(x,s)| 
\end{equation}
that are valid on a fixed time interval 
independent of any lower bounds on $P_+$.

\begin{lemma}\label{lem:Mbound}
There exists $T_\star = T_\star(\eps,h_*)>0$ independent of the initial data,
such that if \eqref{initial energy} holds and also
\begin{equation}\label{ib:M0}
M(0) \leq \frac14 C_3,
\end{equation}
then 
\[
M(t) \leq C_3 \quad\mbox{for all } t\in[0,T_{\max}\wedge T_\star]\,.
\]
\end{lemma}

Taking this result for granted for the moment, let us complete the proof of Theorem~\ref{thm:blowup}.  
We study solutions with smooth initial data that satisfy the relative energy bound
\eqref{initial energy} and the (one-sided) sup bound \eqref{ib:M0}.

We first claim that under a further condition on initial data, necessarily $T_{\max}\leq T_\star$.
We argue as follows.  From \eqref{d:Ppm} we infer that if $T_{\max}>T_\star$, 
then on the time interval $[0,T_\star]$
the norm $\norm{W(\cdot,t)}_{H^2}$ is bounded and hence so is $ \norm{P_\pm}_{L^\infty}$.

From \eqref{e:p plus}, however, using the inequality $P_+P_-\leq \frac18 P_+^2+2P_-^2$ we find that
along any ``+'' characteristic $x=X_+(\xi,t)$,
\begin{equation}
\label{idp:Pp}
\frac{d^+}{dt} P_+ \leq  -\frac{3}{8}P_+^2 + P_+ P_- + \frac{3}{8}P_-^2 + C_2 
\leq -\frac{1}{4}P_+^2 + 3C_3^2 + C_2 \leq -\frac18 P_+^2 \,,
\end{equation}
provided \blue{$\frac18 P_+(\xi,t)^2 \ge 3C_3^2+C_2$ at $t=0$
(for then $P_+^2$ is increasing).}
Choose $\kappa_0<0$ so large that 
\begin{equation} \label{d:kappa0}
\frac18 \kappa_0^2 \ge 3C_3^2+C_2 \quad\text{and}\quad \kappa_0 <  - \frac 8{T_\star}\,,
\end{equation}
and set $\kappa(t) = \kappa_0(1+\frac18 \kappa_0 t)\inv$.
Since $\kappa' = -\frac18\kappa^2$,  
it follows that if $P_+(\xi,0) \leq \kappa_0$ then
\begin{equation}
P_+(\xi,t) \leq \kappa(t) \to -\infty 
\quad \text{as} \quad t \uparrow -{8}\kappa_0\inv < T_\star.
\end{equation}
This proves the following.
\begin{lemma} Necessarily $T_{\max}\leq T_\star$, if initially
\eqref{initial energy} and \eqref{ib:M0} hold, and
\begin{equation}\label{ib:Pp}
\inf_{\xi\in\RR} P_+(\xi,0)< \kappa_0\,.
\end{equation}
\end{lemma}

Now, the essential point is that it is straightforward to construct smooth initial data
that satisfy the required bounds to this point.
We omit the proof of the following.
\begin{lemma} \label{lem:ICs}
There exist smooth initial data $W^0$ of compact support in $\RR$
such that $E_\star$ is arbitrarily small, 
$M(0)$ satisfies \eqref{ib:M0}, and $P_+(\cdot,0)$ satisfies \eqref{ib:Pp}.
\end{lemma}

With any such initial data, it then follows further 
from Corollary~\ref{cor:blowup}, the formulas
\begin{equation}\label{e:uxhx}
u_x = \frac12(P_+ + P_-) \,, \quad
h_x = \sqrt {\frac hg} (P_+ - P_-)\,,
\end{equation}
and Lemma~\ref{lem:Mbound}, 
that $P_+$ cannot remain bounded below and must satisfy
\begin{equation}\label{b:liminfPp}
\liminf_{t\uparrow T_{\max}}  \inf_{x\in\R} P_+(x,t) = -\infty.
\end{equation}

We claim that actually \eqref{lim:Ppblowup} holds, meaning that the ``$\liminf$'' here can be replaced by ``$\lim$.''
The reason for this is that from \eqref{e:p plus} we have
that along any ``+'' characteristic,
\begin{equation}
\label{idm:Pp}
\frac{d^+}{dt} P_+ \geq  -\frac{3}{8}P_+^2 + P_+ P_- + \frac{3}{8}P_-^2 - C_2 
\geq -\frac{1}{2}P_+^2 - 3C_3^2 - C_2 \geq -\frac12 (P_+ + C_4)^2,
\end{equation}
where $\frac12 C_4^2 = 3C_3^2+C_2$. 
By consequence,  if we suppose that \eqref{lim:Ppblowup} is false, 
and instead $\inf_x P(x,t_k)\geq \kappa_1> -\infty$
for some sequence $t_k \to T_{\max}$ in $[0,T_{\max})$, 
then for $k$ so large that 
${1+\frac12\kappa_1 (T_{\max}-t_k)}\ge\frac12$, 
we find by solving the Ricatti inequality above that 
\begin{equation}\label{b:Ppbelow}
\inf_{x\in\RR} P_+(x,t) + C_4 \geq \frac{\kappa_1}
{1+\frac12\kappa_1 (t-t_k)} \ge 2\kappa_1
\quad\text{for all $t\in[t_k,T_{\max})$}.
\end{equation}
This contradicts \eqref{b:liminfPp} and proves \eqref{lim:Ppblowup}.

\textbf{Step Two}.
It remains to prove Lemma~\ref{lem:Mbound}, using a continuation argument.
Set 
\begin{equation}\label{d:T3}
T_3 \defeq \sup\{ t\in[0,T_{\max}): M(t)\leq C_3\} \,.
\end{equation}
Then for $t\in[0,T_3)$ we have the following estimates.  First, as in 
\eqref{idp:Pp} we have
\[
\frac{d^+}{dt} P_+ \leq  -\frac{3}{8}P_+^2 + P_+ P_- + \frac{3}{8}P_-^2 + C_2 
\leq 3C_3^2 + C_2\,,
\]
whence 
\begin{equation}\label{bM:supPp}
\sup_x P_+(x,t) \leq \frac14 C_3 + t(3C_3^2+C_2).
\end{equation}
Similarly, we find that along ``-'' characteristics,
\[
\frac{d^-}{dt} P_- \geq  \frac{3}{8}P_+^2 + P_+ P_- - \frac{3}{8}P_-^2 - C_2 
\geq -3C_3^2 - C_2\,,
\]
whence 
\begin{equation}\label{bM:infPm}
\inf_x P_-(x,t) \geq -\frac14 C_3 - t(3C_3^2+C_2).
\end{equation}
Finally, in a similar way we find
\[
\frac{d^-}{dt} P_- \leq  3P_+^2 + C_2,
\]
whence
\begin{equation}
\label{bM:supPm}
\sup_x P_-(x,t) \leq \frac14 C_3 + t C_2 + 
3 \sup_{\zeta\in\RR} 
\int_0^t P_+^2(X_-(\zeta,s),s)\,ds \,.
\end{equation}

The following estimate on this last integral is the key to the proof.
\begin{lemma}\label{lem:key}
There exists $\Tsstar=\Tsstar(\eps,h_\star)>0$ independent of the initial data,
such that if $T_3< T_{\max}\wedge \Tsstar$ then 
\[
 \sup_{\zeta\in\RR} 
\int_0^t P_+^2(X_-(\zeta,s),s)\,ds \leq \frac 18 C_3 
\quad\text{for all } t\in[0,T_3].
\]
\end{lemma}

Taking this result for granted for the moment, we complete
the proof of Lemma~\ref{lem:Mbound}. 
\begin{proof}[Proof of Lemma~\ref{lem:Mbound}] 
We choose $T_\star>0$ such that 
\begin{equation}
T_\star \leq \Tsstar \quad\text{and}\quad
2T_\star (3C_3^2+C_2) < \frac18 C_3.
\end{equation}
We claim that then $T_3 \geq T_{\max}\wedge T_\star$.
Indeed, if not, then by combining the result of Lemma~\ref{lem:key} with the estimates
in \eqref{bM:supPm}, \eqref{bM:infPm} and \eqref{bM:supPp}, we infer that
\begin{equation}
M(t) \leq \frac78 C_3+ 2t(3C_3^2+C_2) < C_3 \quad 
\text{ for all  $t\in [0,T_3]$.}
\end{equation}
But then by continuity, $M(t)<C_3$ on a larger time interval,
contradicting the definition of $T_3$ in \eqref{d:T3}. 
\end{proof}

\textbf{Step Three}.
Now it remains only to prove Lemma~\ref{lem:key}.

\begin{proof}[Proof of Lemma~\ref{lem:key}]
We first note that due to \eqref{b:h}
the difference between characteristic speeds at the same point satisfies
\begin{equation}\label{e:lamdiff}
\lambda_+ - \lambda_- = 2\sqrt{gh} \ \ \in [ \sqrt{2gh_\star}, \sqrt{6gh_\star}  ] \,.
\end{equation}
Now, suppose $T_3<T_{\max}$, and fix any $\zeta_0\in\RR$. For each $t\in[0,T_3]$, due to 
\eqref{e:lamdiff} there is a unique
$\xi=\xi_0(t)\leq\zeta_0$ such that the ``+'' characteristic starting from
$\xi$ and the ``-'' characteristic starting from $\zeta_0$ intersect 
at time $t$, i.e.,
\begin{equation} \label{e:Xpm meet}
X_+(\xi_0(t),t) = X_-(\zeta_0,t).
\end{equation}
(See Fig.\ref{fig:triangle} for a sketch of the situation.)
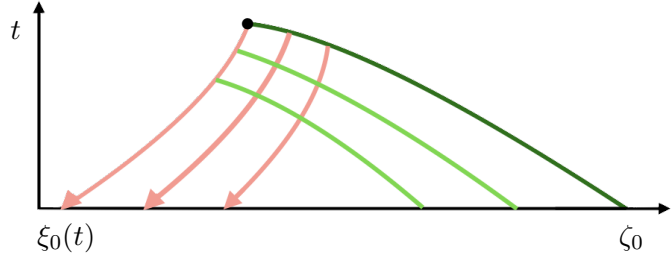
\begin{figure}\label{fig:triangle}
\includegraphics[trim={0 6cm 0 3cm},clip,height=2in]{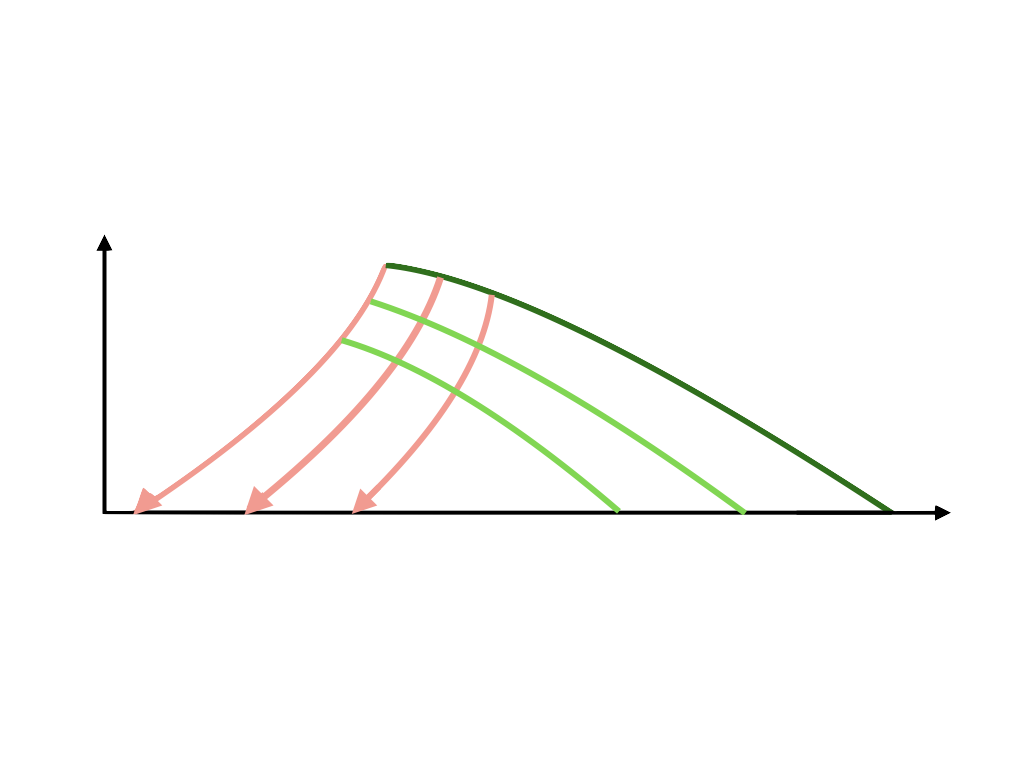}
\put(-260,10){\large $\xi_0(t)$}
\put(-40,10){\large $\zeta_0$}
\put(-270,88){\large $t$}
\put(-183,91){\large $\bullet$} 
\caption{Pullback from $-$ characteristics along $+$ characteristics.
``$\bullet$'' marks the point $(x,t)$ where $x=X_+(\xi_0(t),t) = X_-(\zeta_0,t)$.}
\end{figure}
Note that due to \eqref{def: flow map} and the bound on characteristic speeds in
\eqref{b:lambdapm} we can say that
\begin{align}
\zeta_0-\xi_0(t) &\leq |\zeta_0-X_-(\zeta_0,t)|+|X_+(\xi_0(t),t)-\xi_0(t)|
\nonumber \\ 
&\leq (\|\lambda_+\|_{L^\infty}
+\|\lambda_-\|_{L^\infty})t \leq 4\sqrt{g h_{\max}} t\,.
\label{b:zetaxi}
\end{align}
Differentiating \eqref{e:Xpm meet} in $t$, we find
\[
 \frac{\partial X_+}{\partial\xi}(\xi_0(t), t) \frac{d\xi_0}{dt} + 
\lambda_+(X_+(\xi_0(t), t), t) = 
\lambda_-(X_-(\zeta_0,t), t) 
\]
Due to \eqref{e:Xpm meet} and \eqref{e:lamdiff} it follows
\[
- \frac{\partial X_+}{\partial\xi}(\xi_0(t), t) \frac{d\xi_0}{dt} 
= \lambda_+-\lambda_- = 2\sqrt{gh}.
\]
Now by changing variables $s=s_0(\xi)$ using the inverse function $s_0=\xi_0\inv$,
we get,
\blue{writing $(P_+^2\circ X_+)(\xi,\tau)=P_+(X_+(\xi,\tau),\tau)^2$,}
\begin{align}
&\int_0^t P_+^2(X_-(\zeta_0, s), s)\,ds    \nonumber\\
& \quad =   \int_{\xi_0(t)}^{\zeta_0} \paren{ \frac{P_+^2}{2\sqrt{gh}}  }(X_+(\xi, s_0(\xi)), s_0(\xi)) 
        \frac{\partial X_+}{\partial \xi}(\xi, s_0(\xi)) \,d\xi
        \nonumber\\
&\quad \leq  \frac{1}{\sqrt{2g h_\star}} \int_{\xi_0(t)}^{\zeta_0} P_+^2(X_+(\xi, s_0(\xi)), s_0(\xi))
        \frac{\partial X_+}{\partial \xi}(\xi, s_0(\xi)) \,d\xi   \nonumber\\
&\quad =  \frac{1}{\sqrt{2g h_\star}}  \int_{\xi_0(t)}^{\zeta_0} 
        \paren{ P_+^2(\xi, 0) + \int_0^{s_0(\xi)} 
	\blue{\frac{d}{d\tau}
	\paren{P_+^2\circ X_+\frac{\partial X_+}{\partial \xi}}
        (\xi,\tau)}
	\,d\tau  }\,d\xi \nonumber\\
&\quad \leq  \frac{1}{\sqrt{2g h_\star}} \norm{P_+(\cdot, 0)}_{L^2}^2 
        + \frac{1}{\sqrt{2g h_\star}} \int_{\xi_0(t)}^{\zeta_0} \int_0^{s_0(\xi)} 
        \blue{\frac{d}{d\tau}\paren{P_+^2\circ X_+\frac{\partial X_+}{\partial \xi}}
	} \,d\tau\,d\xi.   \nonumber
\\&\quad \leq  \frac1{16} C_3 + \frac1{\sqrt{2 g h_\star}} \aaa
\label{b:Pp2int}
\end{align}
where
\begin{equation}
\aaa \defeq \int_{\xi_0(t)}^{\zeta_0} \int_0^{s_0(\xi)}
	\blue{\frac{d}{d\tau}
	\paren{P_+^2\circ X_+\frac{\partial X_+}{\partial \xi}}
        (\xi,\tau)}
\,d\tau\,d\xi.
\label{b:A1}
\end{equation}
To get the fourth line in \eqref{b:Pp2int} we use the fundamental theorem of calculus 
along the ``+'' characteristic starting from $(\xi,0)$, and to get the last line we use
\eqref{b:Ppm} and the definition of $C_3$ from \eqref{d:C123}. 

Now, by the crucial derivative computation \eqref{e:plus derivative 0},
since $(\lambda_+)_x =\frac14( 3P_+ + P_-)$ and $|P_-| \leq M(t) \leq C_3$, 
we can bound the the integrand of $\aaa$ by a quadratic polynomial in $P_+$ times 
$\frac{\partial X_+}{\partial \xi}$, in which the linear term can be bounded 
by the quadratic term and the constant term:
\begin{align*}
&
	\blue{\frac{d}{d\tau}
	\paren{P_+^2\circ X_+\frac{\partial X_+}{\partial \xi}}
        (\xi,\tau)}
\leq
\abs{\frac{\partial X_+}{\partial \xi}P_+\paren{3(\lambda_+)_x P_- - 2Q  }}   \nonumber \\
&\quad\leq  \frac{\partial X_+}{\partial \xi}\paren{ \frac{3}{4}\abs{P_+ P_-}\,\abs{3 P_+ + P_-} + 2\abs{P_+}C_2   }      
\nonumber\\ &\quad 
\leq  \frac{\partial X_+}{\partial \xi}  \paren{ 3 C_3P_+^2 + C_3^3 + C_2^2  } \,.
\nonumber 
\end{align*}
Due to the bound \eqref{b:zetaxi},
by using Fubini's theorem we obtain
\begin{align*}
&\int_{\xi_0(t)}^{\zeta_0} \int_0^{s_0(\xi)} \frac{\partial X_+}{\partial \xi}(\xi,\tau) \,d\tau\,d\xi 
 = \int_0^t (\xi_0(\tau)-\xi_0(t))\,d\tau 
\\ &  \qquad
\leq (\zeta_0 - \xi_0(t))t \leq 4\sqrt{gh_{\max}} t^2 < 5\sqrt{gh_\star}t^2.
\end{align*}
To bound the integral of $\frac{\D X_+}{\D \xi}P_+^2$, one uses the inequality
\begin{equation}
\frac{1}{s} \int_0^s f(\tau) \,d\tau \leq f(0) + \int_0^s \abs{f'(\tau)} \,d\tau\quad \forall \,f\in C^1(\RR)
\end{equation}
to obtain
\begin{align}
&\int_{\xi_0}^{\zeta_0} \int_0^{s_0(\xi)} \frac{\partial X_+}{\partial \xi}P_+^2 \,d\tau\,d\xi  \nonumber\\
&\quad \leq  \int_{\xi_0}^{\zeta_0} s_0(\xi) \paren{  P_+^2(\xi, 0) +
         \int_0^{s_0(\xi)}    \abs{
	\blue{\frac{d}{d\tau}
	\paren{P_+^2\circ X_+\frac{\partial X_+}{\partial \xi}}
        (\xi,\tau)}
	 } \,d\tau     } \,d\xi
         \nonumber\\
&\quad \leq  t( \norm{P_+(\cdot, 0)}_{L^2}^2 + \aaa ) \leq t(C_1+\aaa)\,.   
\end{align}
Putting these bounds into \eqref{b:A1}, one obtains
\begin{equation*}
\aaa \leq 3C_3t\paren{ C_1 + \aaa }  + (C_3^3 + C_2^2) 5\sqrt{gh_\star}  t^2 \,.
\end{equation*}

We now choose $\Tsstar>0$ to be so small that
\begin{equation}
3C_3 \Tsstar<\frac12 \quad\text{and}\quad
5(C_3^3+C_2^2)\Tsstar < \frac12  \,.
\end{equation}
Then if $T_3<\Tsstar$, it follows that for all $t\in[0,T_3]$,
\[
\aaa \leq 6 C_3 C_1 t + \sqrt{g h_*}t \,.
\]
Further restricting $\Tsstar$ to be so small that
\[
\frac1{\sqrt{2gh_\star}}\paren{6C_3 C_1+\sqrt{gh_\star}}\Tsstar \leq \frac1{16}C_3,
\]
we can conclude from \eqref{b:Pp2int} that for all $t\in[0,T_3]$,
\begin{equation}
\int_0^t P_+^2(X_-(\zeta_0, s), s)\,ds \leq 
\frac1{16}C_3 +  \frac1{\sqrt{2gh_\star}}\aaa \leq \frac18 C_3.
\end{equation}
This finishes the proof of Lemma~\ref{lem:key}. 
\end{proof}
With this, the proof of Theorem~\ref{thm:blowup} is complete.
\end{proof}

\section{Asymptotic blow-up profile}  \label{s:asym profile}
We recall that for the classical shallow water equations ($\epsl=0$),
the system \eqref{e:r plus}--\eqref{e:r minus}
admits simple wave solutions with $R_-\equiv 0$
and $R_+$ satisfying an inviscid Burgers equation. 
Namely, \eqref{e:r plus} with $\epsl=0$ yields
\begin{equation}
\label{e:Rpzero}
(R_+)_t + \lambda_+ (R_+)_x = 0, \qquad \lambda_+ = \frac34 R_+.
\end{equation}
As is well known (and briefly discussed below) smooth solutions
of this equation with $(R_+)_x<0$ somewhere must break down in finite time,
and typically develop a profile 
with a cube-root singularity at the blow-up point, with
\begin{equation}\label{e:Burgers blowup}
R_+\sim a_0-b_0(x-x_0)^{1/3}.  
\end{equation}
Then after blow-up, 
the singularity changes type
as a shock discontinuity develops.

For the rSV system with $\epsl>0$, 
the coefficients of the quadratic terms in the Ricatti-type system 
\eqref{e:p plus}--\eqref{e:p minus} differ from their values in
the classical system \eqref{e1:Pp}--\eqref{e1:Pm} with $\epsl=0$,
due to an $\epsl$-independent contribution of the local part of 
the nonlocal term.
As we discuss heuristically in this section,
this difference appears to change the nature of the 
typical solution profile at the time of blow-up.
For the blowing-up solutions from section~\ref{s:blowup} above,
we will argue that one should expect that 
the profile near a blow-up point should typically have 
a $\frac35$-root singularity instead:
\begin{equation}\label{e:rSV blowup}
R_+\sim a_\epsl-b_\epsl(x-x_0)^{3/5}.  
\end{equation}
What happens after the blow-up time is not known,
but we may conjecture that solutions develop
$\frac23$-root singularities, like
the weakly singular traveling waves described 
in \cite{pu2018weakly}.

\subsection{Blow-up profile for the rSV equations}
Let us describe heuristically why we may expect the blow-up profile in \eqref{e:rSV blowup}.
Suppose we start close to the blow-up time, taking $P_0(\xi)$ to be initial data for $P_+$ like
that described in the proof of theorem \ref{thm:blowup}, with a large negative minimum at $\xi=0$, say.
In the vicinity of $\xi=0$ we then typically expect quadratic behavior near the minimum, with
\begin{equation}
P_0(\xi) \approx P_0(0) + c_0\xi^2 \ .
\end{equation}
Since $Q$ and $P_-$ are bounded before $P_+$ blows up, we assume
\begin{equation}
\abs{Q}+\abs{P_-} \ll \abs{P_+} \,,
\end{equation}
and neglect these terms, rewriting \eqref{e:p plus} and \eqref{e:plus char.} as
\begin{eqnarray}
\frac{d^+}{dt}P_+ &=& \frac{d}{d t}\left( P_+(X_+(\xi, t), t) \right) 
=  -\frac{3}{8}P_+^2\,,  \label{e:p plus asym.}
\\ 
\paren{\frac{\partial X_+}{\partial \xi} }_t &=& \frac{3}{4}P_+\frac{\partial X_+}{\partial \xi}\,.   \label{e:plus char. asym.}
\end{eqnarray}
With the initial data $P_0$ we solve \eqref{e:p plus asym.} along the ``+'' characteristic curves
to get
\begin{equation}
P_+(X_+(\xi, t), t) = \frac{P_0(\xi)}{1 + \frac{3}{8}tP_0(\xi)}.   \label{e:p plus sol}
\end{equation}
Following the ``+'' characteristic curve emitting from the global minimum point 0 of $P_0$
we expect blow-up to happen first at $\xi=0$ at the time 
$ T = \paren{\frac38|P_0(0)|}\inv \ll 1$, and we find
\begin{equation}
\label{e:PpT}
P_+(X_+(\xi, T), T) \approx -\frac{c_1}{\xi^2}\,,
\qquad c_1 = 
\frac{|P_0(0)|}{\frac38 T c_0}\,.
\end{equation}
 
 Now, from \eqref{e:plus char. asym.} and \eqref{e:p plus asym.} one can compute that
 \begin{equation}
 \frac{d^+}{dt}\paren{\frac{\partial X_+}{\partial \xi}P_+^2} 
 		= 0.
 \label{e:choice of power}
 \end{equation}
(This balancing effect agrees with the rigorous computation \eqref{e:plus derivative 0}.)
Integrating this equation along characteristics up to the blow-up time $T$, we find 
for $\xi$ close to $0$ that 
\begin{equation} \label{e:XxiT}
\frac{\partial X_+}{\partial \xi}(\xi, T) =\frac{P_0^2(\xi)}{P_+^2(X_+(\xi,T), T)} 
= \paren{1 + \frac{3}{8}TP_0(\xi)}^2 = \paren{ \frac38 T c_0}^2 \xi^4.
\end{equation}
Integrating in $\xi$ we get 
\[
X_+(\xi,T)-X_+(0,T) = c_2\xi^5\,. 
\]
Solving for $\xi$ and using this in \eqref{e:PpT} we find, for $x$ near 
$x_0 = X_+(0,T)$,
\[
P_+(x,T) \approx -c_3 (x-x_0)^{-2/5} \,.
\]
Integrating in $x$ now yields \eqref{e:rSV blowup}.

We remark that these heuristics lead us to expect 
that $P_+(\cdot,T)$ belongs to $L^p(\RR)$ for $p<\frac52$.
However, if we repeat the calculations with $P_0$ having a 
degenerate minimum $\sim P_0+c_0\xi^{2n}$ for arbitrary $n\in\N$ we find
that in general $P_+(\cdot,T)$ need not remain in $L^p$ for any $p>2$.

\subsection{Comparison with the inviscid Burgers equation}
We briefly indicate how the calculations above differ with 
the situation when $\epsl=0$.  In this case, the characteristic speed $u=\lambda_+$
in \eqref{e:Rpzero} satisfies the inviscid Burgers equation
\begin{equation}
u_t + uu_x = 0 \,.
\end{equation}
From Burgers equation $v \defeq u_x= \frac34(R_+)_x$ satisfies
$v_t + uv_x = - v^2$,
which implies that along characteristic curves 
$X(\xi, \cdot)$ with $\frac{\partial X}{\partial t}(\xi, t) = u(X(\xi, t), t)$ we have,
analogous to \eqref{e:plus char. asym.} and \eqref{e:p plus asym.},
\[
\frac{d}{dt}\big(v(X(\xi, t), t)\big) = -v^2 \,,
\qquad \paren{\frac{\partial X}{\partial \xi}}_t =  v \frac{\partial X}{\partial \xi}(\xi, t)      \,.
\]
Then similar to \eqref{e:PpT} it follows that 
$v(X(\xi,t),t)\approx -c/\xi^2$ at the time of blow-up.

Differing from \eqref{e:choice of power}, however, we have instead
 \begin{equation}
 \frac{d}{d t}\paren{\frac{\partial X}{\partial \xi}v} = 0 \,,
 \end{equation}
where $v$ appears here and not $v^2$. Now instead of \eqref{e:XxiT}
one finds $\frac{\D X}{\D\xi}\sim c \xi^2$ and $X-x_0\sim c \xi^3$.
From this one deduces that at blow-up,
\begin{equation}
 v(x,t) \sim -c(x-x_0)^{-2/3}, 
\end{equation}
whence \eqref{e:Burgers blowup} follows.

\section*{Acknowledgments} 
The authors are grateful to an anonymous referee for 
indicating how to fix a serious mistake in the original proof 
of the blow-up criterion. 
RLP thanks Robin Ming Chen for discussions regarding the Green-Naghdi equations.
This material is based upon work supported by the National
Science Foundation under grants 
DMS 1514826 and 1812573 (JGL) and
DMS 1515400 and 1812609 (RLP),
partially supported by the Simons Foundation under grant 395796, 
by the Center for Nonlinear Analysis (CNA)
under National Science Foundation PIRE Grant no.\ OISE-0967140,
and by the NSF Research Network Grant no.\ RNMS11-07444 (KI-Net).


\bibliographystyle{siam} 
\bibliography{main2-shallowwater}

\end{document}